\documentclass[12 pt]{amsart}%
\usepackage{pinlabel}
\usepackage[all]{xy}
\usepackage{amscd}
\usepackage{amsfonts}
\usepackage{graphicx}
\usepackage{amsmath}
\usepackage{amssymb}
\usepackage{setspace}
\usepackage{amssymb, amsmath, lmodern}
\usepackage{hyperref}
\usepackage{setspace}
\usepackage[margin=1.2 in]{geometry}
\usepackage[T1]{fontenc}
\usepackage{mathastext}
\usepackage{pdfrender,xcolor}%
\providecommand{\U}[1]{\protect\rule{.1in}{.1in}}
\providecommand{\U}[1]{\protect\rule{.1in}{.1in}}
\providecommand{\U}[1]{\protect\rule{.1in}{.1in}}
\providecommand{\U}[1]{\protect\rule{.1in}{.1in}}
\providecommand{\U}[1]{\protect\rule{.1in}{.1in}}
\providecommand{\U}[1]{\protect\rule{.1in}{.1in}}
\providecommand{\U}[1]{\protect\rule{.1in}{.1in}}
\providecommand{\U}[1]{\protect\rule{.1in}{.1in}}
\providecommand{\U}[1]{\protect\rule{.1in}{.1in}}
\providecommand{\U}[1]{\protect\rule{.1in}{.1in}}
\providecommand{\U}[1]{\protect\rule{.1in}{.1in}}
\providecommand{\U}[1]{\protect\rule{.1in}{.1in}}
\providecommand{\U}[1]{\protect\rule{.1in}{.1in}}
\newtheorem{theorem}{Theorem}

\newtheorem{condition}[theorem]{Condition}

\newtheorem{corollary}[theorem]{Corollary}

\newtheorem{definition}[theorem]{Definition}
\newtheorem{example}[theorem]{Example}

\newtheorem{lemma}[theorem]{Lemma}

\newtheorem{problem}[theorem]{Problem}
\newtheorem{proposition}[theorem]{Proposition}
\newtheorem{remark}[theorem]{Remark}

\begin{document}
\date{\today}
\title[Square-Summable]{Regular Sampling on Metabelian Nilpotent Lie Groups}
\author[V. Oussa]{Vignon S. Oussa}
\address{Dept.\ of Mathematics \& Computer Science\\
Bridgewater State University\\
Bridgewater, MA 02325 U.S.A.\\
 }
\email{vignon.oussa@bridgew.edu}
\keywords{uniform discrete groups}
\subjclass[2000]{22E25, 22E27}

\begin{abstract}
Let $N$ be a simply connected, connected non-commutative nilpotent Lie group
with Lie algebra $\mathfrak{n}$ having rational structure constants. We assume
that $N=P\rtimes M,$ $M$ is commutative, and for all $\lambda\in
\mathfrak{n}^{\ast}$ in general position the subalgebra $\mathfrak{p}=\log(P)$
is a polarization ideal subordinated to $\lambda$ ($\mathfrak{p}$ is a maximal
ideal satisfying $[\mathfrak{p},\mathfrak{p}]\subseteq\ker\lambda$ for all
$\lambda$ in general position and $\mathfrak{p}$ is necessarily commutative.)
Under these assumptions, we prove that there exists a discrete uniform
subgroup $\Gamma\subset N$ such that $L^{2}(N)$ admits band-limited spaces
with respect to the group Fourier transform which are sampling spaces with
respect to $\Gamma.$ We also provide explicit sufficient conditions which are
easily checked for the existence of sampling spaces. Sufficient conditions for
sampling spaces which enjoy the interpolation property are also given. Our
result bears a striking resemblance with the well-known
Whittaker-Kotel'nikov-Shannon sampling theorem.

\end{abstract}
\maketitle




\vskip0.5cm


\vskip0.5cm

\section{Introduction}

It is a well-established fact that a band-limited function on the real line
with its Fourier transform vanishing outside of an interval $\left[
-\frac{\Omega}{2},\frac{\Omega}{2}\right]  $ can be reconstructed by the
Whittaker-Kotel'nikov-Shannon sampling series from its values at points in the
lattice $\frac{1}{\Omega}$ $\mathbb{Z}$ (see \cite{Stens}). This series
expansion takes the form
\[
f\left(  t\right)  =\sum_{k\in\mathbb{Z}}f\left(  \frac{k}{\Omega}\right)
\frac{\sin\left(  \pi\Omega\left(  t-\frac{k}{\Omega}\right)  \right)  }%
{\pi\Omega\left(  t-\frac{k}{\Omega}\right)  }%
\]
with convergence in $L^{2}\left(  \mathbb{R}\right)  $ as well as convergence
in $L^{\infty}\left(  \mathbb{R}\right)  $. A relatively novel problem in
harmonic analysis has been to find analogues of Whittaker-Kotel'nikov-Shannon
sampling series for non-commutative groups. Since $\mathbb{R}$ is a
commutative nilpotent Lie group, it is natural to investigate if it is
possible to extend Whittaker-Kotel'nikov-Shannon's theorem to nilpotent Lie
groups which are not commutative.

Let $G$ be a locally compact group and $\Gamma$ a discrete subset of $G.$ Let
$\mathbf{H}$ be a left-invariant closed subspace of $L^{2}\left(  G\right)  $
consisting of continuous functions. We say that $\mathbf{H}$ is a
\textbf{sampling space} with respect to the set $\Gamma$ \cite{Fuhr cont} if
the following conditions are satisfied. Firstly, the restriction map $f\mapsto
f|_{\Gamma}$ defines a constant multiple of an isometry of $\mathbf{H}$ into
the Hilbert space of square-summable sequences defined over $\Gamma.$ In other
words, there exists a positive constant $c_{\mathbf{H}}$ such that%

\begin{equation}
\sum_{\gamma\in\Gamma}\left\vert f\left(  \gamma\right)  \right\vert
^{2}=c_{\mathbf{H}}\left\Vert f\right\Vert _{2}^{2} \label{sum}%
\end{equation}
for all $f$ in $\mathbf{H}$. Secondly, there exists a vector $s$ in
$\mathbf{H}$ such that an arbitrary element $f$ in the given Hilbert space has
the expansion
\begin{equation}
f\left(  x\right)  =\sum_{\gamma\in\Gamma}f\left(  \gamma\right)  s\left(
\gamma^{-1}x\right)  \label{expansion}%
\end{equation}
with convergence in the norm of $L^{2}\left(  G\right)  .$ If $\Gamma$ is a
discrete subgroup of $G$, we say that $\mathbf{H}$ is a \textbf{regular
sampling space} with respect to $\Gamma.$ Also, if $\mathbf{H}$ is a sampling
space with respect to $\Gamma$ and if the restriction mapping $f\mapsto
f|_{\Gamma}\in l^{2}\left(  \Gamma\right)  $ is surjective then we say that
$\mathbf{H}$ has the \textbf{interpolation property} with respect to $\Gamma.$
This notion of sampling space is taken from \cite{Fuhr cont} and is analogous
to Whittaker-Kotel'nikov-Shannon's theorem. In \cite{FG}, the authors used a
less restrictive definition. They defined a sampling space to be a
left-invariant closed subspace of $L^{2}\left(  G\right)  $ consisting of
continuous functions with the additional requirement that the restriction map
$f\mapsto f|_{\Gamma}$ is a topological embedding of $\mathbf{H}$ into
$l^{2}\left(  \Gamma\right)  $ in the sense that there exist positive real
numbers $a\leq b$ such that
\[
a\left\Vert f\right\Vert _{2}^{2}\leq\sum_{\gamma\in\Gamma}\left\vert f\left(
\gamma\right)  \right\vert ^{2}\leq b\left\Vert f\right\Vert _{2}^{2}
\]
for all $f\in\mathbf{H}.$ The positive number $b/a$ is called the tightness of
the sampling set. Notice that in (\ref{sum}) the tightness of the sampling is
required to be equal to one. Using oscillation estimates, the authors in
\cite{FG} provide general but precise results on the existence of sampling
spaces on locally compact groups. The band-limited vectors in \cite{FG} are
functions that belong to the range of a spectral projection of a self-adjoint
positive definite operator on $L^{2}(G)$ called the sub-Laplacian. This notion
of band-limitation is essentially due to Pesenson \cite{Pes} and does not rely
on the group Fourier transform.

We shall employ in this work a different concept of band-limitation which in
our opinion is consistent with the classical one
(Whittaker-Kotel'nikov-Shannon band-limitation), and the main objective of the
present work is to prove that under reasonable assumptions (see Condition
\ref{cond}) Whittaker-Kotel'nikov-Shannon Theorem naturally extends to a large
class of non-commutative nilpotent Lie groups of arbitrary step.

Let $G$ be a simply connected, connected nilpotent Lie group with Lie algebra
$\mathfrak{g}.$ A subspace $\mathbf{H}$ of $L^{2}(G)$ is said to be a
\textbf{band-limited space} with respect to the group Fourier transform if
there exists a bounded subset $E$ of the unitary dual of the group $G$ such
that $E$ has positive Plancherel measure, and $\mathbf{H}$ consists of vectors
whose group Fourier transforms are supported on the bounded set $E.$ In this
work, we address the following.

\begin{problem}
\label{prob}Let $N$ be a simply connected and connected nilpotent Lie group
with Lie algebra $\mathfrak{n}$ with rational structure constants. Are there
conditions on the Lie algebra $\mathfrak{n}$ under which there exists a
uniform discrete subgroup $\Gamma\subset N=\exp\mathfrak{n}$ such that
$L^{2}\left(  N\right)  $ admits a band-limited (in terms of the group Fourier
transform) sampling subspace with respect to $\Gamma?$
\end{problem}

Firstly, we observe that if $N=\mathbb{R}^{d}$ then $\Gamma$ can be taken to
be an integer lattice, and the Hilbert space of functions vanishing outside
the cube $\left[  -\frac{1}{2},\frac{1}{2}\right]  ^{d}$ is a sampling space
which enjoys the interpolation property with respect to $\mathbb{Z}^{d}$.
Secondly, let $N$ be the Heisenberg Lie group and let $\Gamma$ be a discrete
uniform subgroup of $N$ which we realize as follows:
\[
N=\left\{  \left[
\begin{array}
[c]{ccc}%
1 & x & z\\
0 & 1 & y\\
0 & 0 & 1
\end{array}
\right]  :x,y,z\in\mathbb{R}\right\}  \text{ and }\Gamma=\left\{  \left[
\begin{array}
[c]{ccc}%
1 & m & k\\
0 & 1 & l\\
0 & 0 & 1
\end{array}
\right]  :k,l,m\in\mathbb{Z}\right\}  .
\]
It is shown in \cite{Fuhr cont, Currey} that there exist subspaces of
$L^{2}\left(  N\right)  $ which are sampling subspaces with respect to
$\Gamma$. We have also established in \cite{oussa,oussa1,oussa2} the existence
of sampling spaces defined over a class of simply connected, connected
nilpotent Lie groups which satisfy the following conditions: $N$ is a step-two
nilpotent Lie group with Lie algebra $\mathfrak{n}$ of dimension $n$ such that
$\mathfrak{n=a\oplus b\oplus c}$ where $\left[  \mathfrak{a},\mathfrak{b}%
\right]  \subseteq\mathfrak{c,}$ $\mathfrak{a},\mathfrak{b}$ are commutative
Lie algebras, $\mathfrak{a}=\mathbb{R}\text{-span}\left\{  X_{1},X_{2}%
,\cdots,X_{d}\right\}  ,\mathfrak{b}=\mathbb{R}\text{-span}\left\{
Y_{1},Y_{2},\cdots,Y_{d}\right\}  ,$ $\mathfrak{c}=\mathbb{R}\text{-span}%
\left\{  Z_{1},Z_{2},\cdots,Z_{n-2d}\right\}  $ ($d\geq1,n>2d$) and
\begin{equation}
\left(  Z_{1},\cdots,Z_{n-2d}\right)  \mapsto\det\left[
\begin{array}
[c]{ccc}%
\left[  X_{1},Y_{1}\right]  & \cdots & \left[  X_{1},Y_{d}\right] \\
\vdots & \ddots & \vdots\\
\left[  X_{d},Y_{1}\right]  & \cdots & \left[  X_{d},Y_{d}\right]
\end{array}
\right]  \label{matrix}%
\end{equation}
is a non-vanishing polynomial in the variables $Z_{1},\cdots,Z_{n-2d}.$ To the
best of our knowledge, prior to this work, regular sampling for band-limited
(in terms of the group Fourier transform) left-invariant spaces defined over
nilpotent Lie groups has only been systematically studied on step one (the
classical Euclidean case) and some step two nilpotent Lie groups \cite{Currey,
Fuhr cont, oussa, oussa2, oussa1}. We shall prove that for any given natural
number $k,$ there exists a nilpotent Lie group of step $k$ which admits
band-limited sampling spaces in terms of the Plancherel transform with respect
to a discrete uniform subgroup (see Example \ref{dimfour}).

\subsection{Overview of the Paper}

Let us start by fixing notation and by recalling some relevant concepts.

\begin{itemize}
\item Let $Q$ be a linear operator acting on an $n$-dimensional real vector
space $V.$ The norm of the matrix $Q$ induced by the max-norm of the vector
space $V$ is given by
\[
\left\Vert Q\right\Vert _{\infty}=\sup\left\{  \left\Vert Qx\right\Vert
_{\max}:x\in V\text{ and }\left\Vert x\right\Vert _{\max}=1\right\}
\]
and the max-norm of an arbitrary vector is given by
\[
\left\Vert x\right\Vert _{\max}=\max\left\{  \left\vert x_{k}\right\vert
:1\leq k\leq n\right\}  .
\]
Next, letting $\left[  Q\right]  $ be the matrix representation of $Q$ with
respect to a fixed basis, the transpose of this matrix is denoted $\left[
Q\right]  ^{T}.$

\item Given a countable sequence $(f_{i})_{i\in I}$ of vectors in a Hilbert
space $\mathbf{H},$ we say that$(f_{i})_{i\in I}$ forms a \textbf{frame }
\cite{Pete, Han Yang Wang, Pfander} if and only if there exist strictly
positive real numbers $a,b$ such that for any vector $f\in\mathbf{H}$,%
\[
a\left\Vert f\right\Vert ^{2}\leq\sum_{i\in I}\left\vert \left\langle
f,f_{i}\right\rangle \right\vert ^{2}\leq b\left\Vert f\right\Vert ^{2}.
\]
In the case where $a=b$, the sequence $(f_{i})_{i\in I}$ is called a
\textbf{tight frame}. If $a=b=1$, $(f_{i})_{i\in I}$ is called a
\textbf{Parseval frame}.

\item Let $\pi$ be a unitary representation of a locally compact group $G$
acting on a Hilbert space $\mathbf{H}_{\pi}.$ We say that the representation
$\pi$ is \textbf{admissible} \cite{Fuhr cont} if there exists a vector $h$ in
$\mathbf{H}_{\pi}$ such that the linear map
\begin{equation}
g\mapsto V_{h}^{\pi}\left(  f\right)  =\left\langle f,\pi\left(  g\right)
h\right\rangle \label{admissibility}%
\end{equation}
defines an isometry of the Hilbert space $\mathbf{H}_{\pi}$ into $L^{2}\left(
G\right)  .$ In this case, the vector $h$ is called an \textbf{admissible
vector} for the representation $\pi.$

\item Let $\left(  A,\mathcal{M}\right)  $ be a measurable space. A family
$\left(  \mathbf{H}_{a}\right)  _{a\in A}$ of Hilbert spaces indexed by the
set $A$ is called a field of Hilbert spaces over $A$ \cite{Folland}. An
element $f$ of $\Pi_{a\in A}\mathbf{H}_{a}$ is a vector-valued function
$a\mapsto f\left(  a\right)  \in$ $\mathbf{H}_{a}$ defined on the set $A.$
Such a map is called a \textbf{vector field} on $A.$ A \textbf{measurable
field }of Hilbert spaces defined on a measurable set $A$ is a field of Hilbert
spaces together with a countable set $\left(  e_{j}\right)  _{j\in J}$ of
vector fields such that the functions $a\mapsto\left\langle e_{j}\left(
a\right)  ,e_{k}\left(  a\right)  \right\rangle _{\mathbf{H}_{a}}$ are
measurable for all $j,k\in J,$ and the linear span of $\left\{  e_{j}\left(
a\right)  \right\}  _{j\in J}$ is dense in $\mathbf{H}_{a}$ for each $a.$ A
vector field $f$ is called a \textbf{measurable vector field} if
$a\mapsto\left\langle f\left(  a\right)  ,e_{j}\left(  a\right)  \right\rangle
_{\mathbf{H}_{a}}$ is a measurable function for each index $j.$

\item Let $\mathfrak{n}$ be a nilpotent Lie algebra of dimension $n$, and let
$\mathfrak{n}^{\ast}$ be the dual vector space of $\mathfrak{n.}$ A
\textbf{polarizing subalgebra }$\mathfrak{p}\left(  \lambda\right)  $
subordinated to a linear functional $\lambda\in\mathfrak{n}^{\ast}$ (see
\cite{Corwin, vignon}) is a maximal algebra satisfying
\[
\left[  \mathfrak{\mathfrak{p}}\left(  \lambda\right)  \mathfrak{,\mathfrak{p}%
}\left(  \lambda\right)  \right]  =\text{\textrm{Span-}}\left\{  \left[
X,Y\right]  \in\mathfrak{n}:X,Y\in\mathfrak{\mathfrak{p}}\left(
\lambda\right)  \right\}  \subseteq\ker\left(  \lambda\right)  .
\]

\item The coadjoint action on the dual of $\mathfrak{n}$ is the dual of the
adjoint action of $N=\exp\mathfrak{n}$ on $\mathfrak{n}$. In other words, for
$X\in\mathfrak{n},$ and a linear functional $\lambda\in\mathfrak{n}^{\ast}$,
the coadjoint action is defined as follows:
\begin{equation}
\left(  \exp X\cdot\lambda\right)  \left(  Y\right)  =\left\langle \left(
e^{ad-X}\right)  ^{\ast}\lambda,Y\right\rangle =\left[  \left(  e^{ad-X}%
\right)  ^{\ast}\lambda\right]  \left(  Y\right)  . \label{coadjoint}%
\end{equation}

\end{itemize}

The following is a concept which is central to our results.

\begin{definition}
Let $\mathfrak{p}$ be a subalgebra of $\mathfrak{n}.$ We say that
$\mathfrak{p}$ is a \textbf{constant polarization subalgebra} of
$\mathfrak{n}$ if there exists a Zariski open set $\Omega\subset
\mathfrak{n}^{\ast}$ which is invariant under the coadjoint action of $N$ and
$\mathfrak{p}$ is a polarization subalgebra subordinated to every linear
functional in $\Omega.$
\end{definition}

In other words, $\mathfrak{p}$ is a constant polarization subalgebra of
$\mathfrak{n}$ if $\mathfrak{p}$ is a polarization algebra for all linear
functionals in general position, and it can then be shown (see Proposition
\ref{commutative}) that $\mathfrak{p}$ is necessarily commutative.

\subsubsection{Summary of Main Results}

Let us suppose that $N=P\rtimes M=\exp\left(  \mathfrak{p}\right)  \rtimes
\exp\left(  \mathfrak{m}\right)  $ is a simply connected, connected
non-commutative nilpotent Lie group with Lie algebra $\mathfrak{n=p\oplus m}$
such that

\begin{condition}
\label{cond} \ \ \ \ \ \ \ \ \ \ \ \ 

\begin{enumerate}
\item $\mathfrak{p}$ is a constant polarization ideal of $\mathfrak{n}$ (thus
commutative) $\mathfrak{m}$ is commutative as well, $p=\dim\mathfrak{p}\text{,
}m=\dim\mathfrak{m}$ and $p-m>0.$

\item There exists a strong Malcev basis $\left\{  Z_{1},\cdots,Z_{p}%
,A_{1},\cdots,A_{m}\right\}  $ for $\mathfrak{n}$ such that $\left\{
Z_{1},\cdots,Z_{p}\right\}  $ is a basis for $\mathfrak{p}$ and $\left\{
A_{1},\cdots,A_{m}\right\}  $ is a basis for $\mathfrak{m}$ and
\[
\Gamma=\exp\left(  \mathbb{Z}Z_{1}+\cdots+\mathbb{Z}Z_{p}\right)  \exp\left(
\mathbb{Z}A_{1}+\cdots+\mathbb{Z}A_{m}\right)
\]
is a discrete uniform subgroup of $N.$ This is equivalent to the fact that
$\mathfrak{n}$ has rational structure constants (see Chapter $5,$
\cite{Corwin}).
\end{enumerate}
\end{condition}


In order to properly introduce the concept of band-limitation with respect to
the group Fourier transform, we appeal to Kirillov's theory \cite{Corwin}
which states that the unitary irreducible representations of $N$ are
parametrized by orbits of the coadjoint action of $N$ on the dual of its Lie
algebra and can be modeled as acting in $L^{2}(\mathbb{R}^{m}).$ Let $\Sigma$
be a parameterizing set for the unitary dual of $N.$ In other words, $\Sigma$
is a cross-section for the coadjoint orbits in an $N$-invariant Zariski open
set $\Omega\subset\mathfrak{n}^{\ast}.$ If the ideal $\mathfrak{p}$ is a
constant polarization for $\mathfrak{n}$ then the orbits in general position
are $2m$-dimensional submanifolds of $\mathfrak{n}^{\ast}$ and we shall (this
is a slight abuse of notation) regard $\Sigma$ as a Zariski open subset of
$\mathbb{R}^{p-m}=\mathbb{R}^{n-2m}$. Next, let $L$ be the left regular
representation of $N$ acting on $L^{2}\left(  N\right)  $ by left
translations. Let
\[
\mathcal{P}:L^{2}\left(  N\right)  \xrightarrow{\hspace*{1cm}}L^{2}\left(
\Sigma\mathbf{,}\text{ }L^{2}\left(  \mathbb{R}^{m}\right)  \otimes
L^{2}\left(  \mathbb{R}^{m}\right)  ,d\mu\left(  \lambda\right)  \right)
\]
be the Plancherel transform which defines a unitary map on $L^{2}\left(
N\right)  $ (see Subsection \ref{Plancherel}). The Plancherel transform
intertwines the left regular representation with a direct integral of
irreducible representations of $N.$ The measure used in the decomposition is
the so-called Plancherel measure: $d\mu;$ which is a weighted Lebesgue measure
on $\Sigma.$ More precisely $d\mu\left(  \lambda\right)  $ is equal to
$\left\vert \mathbf{P}\left(  \lambda\right)  \right\vert d\lambda$ where
$\mathbf{P}\left(  \lambda\right)  $ is a polynomial defined over $\Sigma$ and
$d\lambda$ is the Lebesgue measure on $\Sigma$ (see Lemma
\ref{Plancherel measure}.) Given a $\mu$-measurable bounded set $\mathbf{A}%
\subset\Sigma$, and a measurable field of unit vectors $\left(  \mathbf{u}%
\left(  \lambda\right)  \right)  _{\lambda\in\mathbf{A}}$ in $L^{2}\left(
\mathbb{R}^{m}\right)  ,$ the Hilbert space $\mathbf{H}_{\mathbf{A}}$ which
consists of vectors $f\in L^{2}(N)$ such that
\[
\mathcal{P}f\left(  \lambda\right)  =\left\{
\begin{array}
[c]{c}%
\mathbf{v}\left(  \lambda\right)  \otimes\mathbf{u}\left(  \lambda\right)
\text{ if }\lambda\in\mathbf{A}\\
0\otimes0\text{ if }\lambda\notin\mathbf{A}%
\end{array}
\right.
\]
and $\left(  \mathbf{v}\left(  \lambda\right)  \otimes\mathbf{u}\left(
\lambda\right)  \right)  _{\lambda\in\mathbf{A}}$ is a measurable field of
rank-one operators is a left-invariant multiplicity-free band-limited subspace
of $L^{2}\left(  N\right)  $ which we identify with $L^{2}(\mathbf{A}%
\times\mathbb{R}^{m}).$ Conjugating the operators $L(x)$ by the Plancherel
transform, we obtain that
\[
\lbrack\mathcal{P}\circ L(x)\circ\mathcal{P}^{-1}](\mathbf{v}_{\lambda}%
\otimes\mathbf{u}_{\lambda})_{\lambda\in\Sigma}=([\sigma_{\lambda
}(x)\mathbf{v}_{\lambda}]\otimes\mathbf{u}_{\lambda})_{\lambda\in\Sigma}%
\equiv\sigma_{\lambda}(x)\mathbf{v}(\lambda,\cdot)
\]
where $\sigma_{\lambda}$ is the unitary irreducible representation
corresponding to the linear functional $\lambda\in\Sigma.$ Let $L_{\mathbf{H}%
_{\mathbf{A}}}$ be the representation induced by the action of the left
regular representation on the Hilbert space $\mathbf{H}_{\mathbf{A}}.$ It can
be shown that if the spectral set $\mathbf{A}$ satisfies precise conditions
specified in Theorem \ref{Main 2} then the restriction of $L_{\mathbf{H}%
_{\mathbf{A}}}$ to the discrete group $\Gamma$ is unitarily equivalent with a
subrepresentation of the left regular representation of $\Gamma$ acting on
$l^{2}\left(  \Gamma\right)  .$ The existence of band-limited sampling spaces
with respect to $\Gamma$ can then be established by directly appealing to
known results contained in the Monograph \cite{Fuhr cont}. Define
$\beta:\Sigma\times\mathbb{R}^{m}\rightarrow\mathbb{R}^{p}$ such that
\[
\beta\left(  \lambda,t\right)  =\exp\left(  t_{1}A_{1}+\cdots+t_{m}%
A_{m}\right)  \cdot\lambda|\mathfrak{p}^{\ast}
\]
where $t=\left(  t_{1},\cdots,t_{m}\right)  $. Under the assumptions listed in
Condition \ref{cond}, it is worth noting that $\beta$ is a diffeomorphism
(Lemma \ref{beta 1}) and the following holds true.

\begin{theorem}
\label{Main 2}Let $N=PM=\exp\left(  \mathfrak{p}\right)  \exp\left(
\mathfrak{m}\right)  $ be a simply connected, connected nilpotent Lie group
with Lie algebra $\mathfrak{n}$ satisfying Condition \ref{cond}. Let
$\mathbf{A}$ be a $\mu$-measurable bounded subset of $\Sigma.$

\begin{enumerate}
\item If $\beta\left(  \mathbf{A}\times\left[  0,1\right)  ^{m}\right)  $ has
positive Lebesgue measure in $\mathbb{R}^{p}$ and is contained in a
fundamental domain of $\mathbb{Z}^{p}$ then there exists a vector $\eta
\in\mathbf{H}_{\mathbf{A}}$ such that $V_{\eta}^{L}\left(  \mathbf{H}%
_{\mathbf{A}}\right)  $ is a left-invariant subspace of $L^{2}\left(
N\right)  $ which is a sampling space with respect to $\Gamma.$

\item If $\beta\left(  \mathbf{A}\times\left[  0,1\right)  ^{m}\right)  $ is
equal to a fundamental domain of $\mathbb{Z}^{p}$ then there exists a vector
$\eta\in\mathbf{H}_{\mathbf{A}}$ such that $V_{\eta}^{L}\left(  \mathbf{H}%
_{\mathbf{A}}\right)  $ is a left-invariant subspace of $L^{2}\left(
N\right)  $ which is a sampling space with the interpolation property with
respect to $\Gamma.$
\end{enumerate}
\end{theorem}

Let $s=\left(  s_{1},s_{2},\cdots,s_{m}\right)  $ be an element of
$\mathbb{R}^{m}$ and define $A\left(  s\right)  $ to be the restriction of the
linear map $ad\left(  -\sum_{j=1}^{m}s_{j}A_{j}\right)  $ to the ideal
$\mathfrak{p}\subset\mathfrak{n.}$ Let $\left[  A\left(  s\right)  \right]  $
be the matrix representation of the linear map $A\left(  s\right)  $ with
respect to the basis $\left\{  Z_{1},\cdots,Z_{p}\right\}  $. Let $e^{[A(s)]}$
be the matrix obtained by exponentiating $\left[  A\left(  s\right)  \right]
.$ Since $s\mapsto\left\Vert e^{\left[  A\left(  s\right)  \right]  ^{T}%
}\right\Vert _{\infty}$ is a continuous function of $s$ it is bounded over any
compact set and in particular over the cube $[0,1]^{m}.$ As such, letting
$\varepsilon$ be a positive real number satisfying
\begin{equation}
\varepsilon\leq\delta=\frac{1}{2}\left(  \sup\left\{  \left\Vert e^{\left[
A\left(  s\right)  \right]  ^{T}}\right\Vert _{\infty}:s\in\left[  0,1\right)
^{m}\right\}  \right)  ^{-1}<\infty, \label{above}%
\end{equation}
we shall prove that under the assumptions provided in Condition \ref{cond},
the set
\[
B(\varepsilon)=\beta\left(  \left(  -\varepsilon,\varepsilon\right)
^{n-2m}\times\left[  0,1\right)  ^{m}\right)
\]
has positive Lebesgue measure and is contained in a fundamental domain of
$\mathbb{Z}^{p}.$ Appealing to Theorem \ref{Main 2}, we are then able to
establish the following result which provides us with a concrete formula for
the bandwidth of various sampling spaces.

\begin{corollary}
\label{Main1}Let $N=PM=\exp\left(  \mathfrak{p}\right)  \exp\left(
\mathfrak{m}\right)  $ be a simply connected, connected nilpotent Lie group
with Lie algebra $\mathfrak{n}$ satisfying Condition \ref{cond}. For any
positive number $\varepsilon$ satisfying (\ref{above}) there exists a
band-limited vector $\eta=\eta_{\varepsilon}\ $in the Hilbert space
$\mathbf{H}_{\left(  -\varepsilon,\varepsilon\right)  ^{n-2m}}$ such that
$V_{\eta}^{L}\left(  \mathbf{H}_{\left(  -\varepsilon,\varepsilon\right)
^{n-2m}}\right)  $ is a left-invariant subspace of $L^{2}\left(  N\right)  $
which is a sampling space with respect to $\Gamma$.
\end{corollary}

Next, we exhibit several examples to illustrate that the class of groups under
consideration is fairly large.

\begin{example}
\label{dimfour}\text{ }

\begin{enumerate}
\item Let $N$ be a simply connected, connected nilpotent Lie group with Lie
algebra $\mathfrak{n}$ of dimension four or less. Then, there exists a uniform
discrete subgroup $\Gamma\subset N$ such that $L^{2}\left(  N\right)  $ admits
a band-limited sampling subspace with respect to $\Gamma.$ Additionally, the
Heisenberg Lie group admits a sampling space which has the interpolation
property with respect to a uniform discrete subgroup.

\item Let $N$ be a simply connected, connected nilpotent Lie group with Lie
algebra spanned by $Z_{1},Z_{2},\cdots,Z_{p},A_{1},$ the vector space
generated by $Z_{1},Z_{2},\cdots,Z_{p}$ is a commutative ideal, $\left.
\left[  adA_{1}\right]  \right\vert _{\mathfrak{p}}=A$ is a nonzero rational
upper triangular nilpotent matrix of order $p$, and $e^{A}\mathbb{Z}%
^{p}\subseteq\mathbb{Z}^{p}.$ Then $L^{2}\left(  N\right)  $ admits a
band-limited sampling subspace with respect to the discrete uniform subgroup
\[
\exp\left(  \mathbb{Z}Z_{1}+\cdots+\mathbb{Z}Z_{p}\right)  \exp\left(
\mathbb{Z}A_{1}\right)  .
\]

\item Let $N$ be a simply connected, connected nilpotent Lie group with Lie
algebra spanned by $Z_{1},Z_{2},\cdots,Z_{p},A_{1},\cdots,A_{m}$ where
$p=m+1,$ the vector space generated by $Z_{1},Z_{2},\cdots,Z_{p}$ is a
commutative ideal, the vector space generated by $A_{1},\cdots A_{m}$ is
commutative and the matrix representation of $ad\left(  \sum_{k=1}^{m}%
t_{k}A_{k}\right)  $ restricted to $\mathfrak{p}$ is given by
\[
A\left(  t\right)  =\left.  \left[  ad\sum_{k=1}^{m}t_{k}A_{k}\right]
\right\vert \mathfrak{p}=m!\left[
\begin{array}
[c]{cccccc}%
0 & t_{1} & t_{2} & \cdots & t_{m-1} & t_{m}\\
& 0 & t_{1} & t_{2} & \ddots & t_{m-1}\\
&  & 0 & t_{1} & \ddots & \vdots\\
&  &  & 0 & \ddots & t_{2}\\
&  &  &  & \ddots & t_{1}\\
&  &  &  &  & 0
\end{array}
\right]  .
\]
Then $L^{2}\left(  N\right)  $ admits a band-limited sampling subspace with
respect to the discrete uniform subgroup $\exp\left(  \mathbb{Z}Z_{1}%
+\cdots+\mathbb{Z}Z_{p}\right)  \exp\left(  \mathbb{Z}A_{1}+\cdots
+\mathbb{Z}A_{m}\right)  .$
\end{enumerate}
\end{example}

The work is organized as follows. In Section \ref{general}, we present general
well-known results of harmonic analysis on nilpotent Lie groups. Section
\ref{intermediateresult} contains intermediate results leading to the proofs
of Theorem \ref{Main 2}, Corollary \ref{Main1} and Example \ref{dimfour} which
are given in Section \ref{mainresults}. Finally, in Section
\ref{Othersamplingset} we provide a method for constructing other sampling
sets from $\Gamma$ by using automorphisms of the Lie group $N.$

\section{Harmonic Analysis on Nilpotent Lie Groups}

\label{general}

\subsection{Parametrization of Coadjoint Orbits}

Let $\mathfrak{n}$ be a finite-dimensional nilpotent Lie algebra of dimension
$n$. We say that $\mathfrak{n}$ has a rational structure \cite{Corwin} if
there is a real basis $\left\{  Z_{1},\cdots,Z_{n}\right\}  $ for the Lie
algebra $\mathfrak{n}$ having rational structure constants and the rational
span of the basis $\left\{  Z_{1},\cdots,Z_{n}\right\}  $ denoted by
$\mathfrak{n}_{\mathbb{Q}}$ provides a rational structure such that
$\mathfrak{n}$ is isomorphic to the vector space $\mathfrak{n}_{\mathbb{Q}%
}\otimes\mathbb{R}$. Let $\mathfrak{B}=\left\{  Z_{1},\cdots,Z_{n}\right\}  $
be a basis for the Lie algebra $\mathfrak{n}$ such that for any $Z_{i}%
,Z_{j}\in\mathfrak{B},$ we have:%
\[
\left[  Z_{i},Z_{j}\right]  =\sum_{k=1}^{n}c_{ijk}Z_{k}%
\]
and $c_{ijk}\in\mathbb{Q}.$ We say that $\mathfrak{B}$ is a \textbf{strong
Malcev basis }(see Page $10,$ \cite{Corwin}\textbf{)} if and only if for each
$1\leq j\leq n$ the real span of $\left\{  Z_{1},Z_{2},\cdots,Z_{j}\right\}  $
is an ideal of $\mathfrak{n.}$ Now, let $N$ be a connected, simply connected
nilpotent Lie group with Lie algebra $\mathfrak{n}$ having a rational
structure. The following result is taken from Corollary $5.1.10$,
\cite{Corwin}. Let $\left\{  Z_{1},\cdots,Z_{n}\right\}  $ be a strong Malcev
basis for the Lie algebra $\mathfrak{n.}$ There exists a suitable integer $q$
such that $\Gamma_{q}=\exp\left(  q\mathbb{Z}Z_{1}\right)  \cdots\exp\left(
q\mathbb{Z}Z_{n}\right)  $ is a discrete uniform subgroup of $N$ (there is a
compact set $K\subset G$ such that $\Gamma K=N$). Setting $X_{k}=qZ_{k}$ for
$1\leq k\leq n$, from now on, we fix $\left\{  X_{1},\cdots,X_{n}\right\}  $
as a strong Malcev basis for the Lie algebra $\mathfrak{n}$ such that
\[
\Gamma=\exp\left(  \mathbb{Z}X_{1}\right)  \cdots\exp\left(  \mathbb{Z}%
X_{n}\right)
\]
is a discrete uniform subgroup of $N$.

We shall next discuss the Plancherel theory for $N.$ This theory is well
exposed in \cite{Corwin} for nilpotent Lie groups. Let $\mathfrak{s}$ be a
subset of $\mathfrak{n}=\log(N).$ For each linear functional $\lambda
\in\mathfrak{n}^{\ast},$ we define the corresponding set
\[
\mathfrak{s}\left(  \lambda\right)  =\left\{  Z\in\mathfrak{n}:\text{ }%
\lambda\left(  \left[  Z,X\right]  \right)  =0\text{ for every }%
X\in\mathfrak{s}\right\}  .
\]
Next, we consider a fixed strong Malcev basis $\mathcal{B}^{\prime}=\left\{
X_{1},\cdots,X_{n}\right\}  $ $\ $and we construct a sequence of ideals
$\mathfrak{n}_{1}\subseteq\mathfrak{n}_{2}\subseteq\cdots\subseteq
\mathfrak{n}_{n-1}\subseteq\mathfrak{n}$ where each ideal $\mathfrak{n}_{k}$
is spanned by $\left\{  X_{1},\cdots,X_{k}\right\}  .$ It is easy to see that
the differential of the coadjoint action on $\lambda$ at the identity is given
by the matrix
\[
\left[  \left\langle \lambda,\left[  X_{j},X_{k}\right]  \right\rangle
\right]  _{1\leq j,k\leq n}=\left[  \lambda\left(  \left[  X_{j},X_{k}\right]
\right)  \right]  _{1\leq j,k\leq n}.
\]
Defining the skew-symmetric matrix-valued function
\begin{equation}
\lambda\mapsto\mathbf{M}\left(  \lambda\right)  =\left[
\begin{array}
[c]{ccc}%
\lambda\left[  X_{1},X_{1}\right]  & \cdots & \lambda\left[  X_{1}%
,X_{n}\right] \\
\vdots & \ddots & \vdots\\
\lambda\left[  X_{n},X_{1}\right]  & \cdots & \lambda\left[  X_{n}%
,X_{n}\right]
\end{array}
\right]  \label{Mlambda}%
\end{equation}
on $\mathfrak{n}^{\ast},$ it is worth noting that $\mathfrak{n}\left(
\lambda\right)  $ is equal to the null-space of $\mathbf{M}\left(
\lambda\right)  $, if $\mathbf{M}\left(  \lambda\right)  $ is regarded as a
linear operator acting on $\mathfrak{n}$ \cite{vignon}. According to the orbit
method \cite{Corwin}, the unitary dual of $N$ is in one-to-one correspondence
with the set of coadjoint orbits in the dual of the Lie algebra. For each
$\lambda\in\mathfrak{n}^{\ast}$ we define
\begin{equation}
\mathbf{e}\left(  \lambda\right)  =\left\{  1\leq k\leq n:\mathfrak{n}%
_{k}\text{ }\nsubseteq\text{ }\mathfrak{n}_{k-1}+\mathfrak{n}\left(
\lambda\right)  \right\}  . \label{jumping}%
\end{equation}
The set $\mathbf{e}\left(  \lambda\right)  $ collects all basis elements
$\left\{  X_{i}:i\in\mathbf{e}\left(  \lambda\right)  \right\}  \subset
\left\{  X_{1},X_{2},\cdots,X_{n-1},X_{n}\right\}  $ such that if the elements
are ordered such that $\mathbf{e}\left(  \lambda\right)  =\left\{
\mathbf{e}_{1}\left(  \lambda\right)  <\cdots<\mathbf{e}_{2m}\left(
\lambda\right)  \right\}  $ then the dimension of the manifold $\exp\left(
\mathbb{R}X_{\mathbf{e}_{1}\left(  \lambda\right)  }\right)  \cdots\exp\left(
\mathbb{R}X_{\mathbf{e}_{2m}\left(  \lambda\right)  }\right)  \cdot\lambda$ is
equal to the dimension of the $N$-orbit of $\lambda$. Each element of the set
$\mathbf{e}\left(  \lambda\right)  $ is called a \textbf{jump index} and
clearly the cardinality of the set of jump indices $\mathbf{e}\left(
\lambda\right)  $ must be equal to the dimension of the coadjoint orbit of
$\lambda.$

\vskip0.2cm\noindent For each subset $\mathbf{e}^{\circ}$ $\subseteq\left\{
1,2,\cdots,n\right\}  ,$ the set
\[
\Omega_{\mathbf{e}^{\circ}}=\left\{  \lambda\in\mathfrak{n}^{\ast}%
:\mathbf{e}\left(  \lambda\right)  =\mathbf{e}^{\circ}\right\}
\]
is algebraic and $N$-invariant \cite{Currey Can}. Moreover, there exists a set
of jump indices $\mathbf{e}$ such that $\Omega_{\mathbf{e}}=\Omega$ is a
Zariski open set in $\mathfrak{n}^{\ast}$ which is invariant under the action
of $N$ (Theorem $3.1.6,$ \cite{Corwin}.)

\vskip 0.2cm\noindent Put $\Omega=\Omega_{\mathbf{e}}.$ We recall that a
polarization subalgebra subordinated to the linear functional $\lambda$ is a
maximal subalgebra $\mathfrak{p}(\lambda)$ of $\mathfrak{n}^{\ast}$ satisfying
the condition $[\mathfrak{p}(\lambda),\mathfrak{p}(\lambda)]\subseteq
\ker\lambda.$ Notice that if $\mathfrak{p}(\lambda)$ is a polarization
subalgebra associated with the linear functional $\lambda$ then $\chi(\exp
X)=e^{2\pi i\lambda(X)}$ defines a character on $\exp(\mathfrak{p}(\lambda)$.
It is also well-known that $\dim\left(  {\mathfrak{n}}(\lambda)\right)  =n-2m$
and $\dim\left(  \mathfrak{n}/\mathfrak{p}(\lambda)\right)  =m,$ and
$\mathfrak{p}(\lambda)=\sum_{k=1}^{n}\mathfrak{n}_{k}\left(  \lambda
|\mathfrak{n}_{k}\right)  $ (see Page $30,$ \cite{Corwin} and \cite{vignon}.)

\begin{proposition}
\label{commutative} If $\mathfrak{p}$ is a constant polarization for
$\mathfrak{n}$ then it must be commutative.
\end{proposition}

\begin{proof}
Let $\Omega$ be a Zariski open and $N$-invariant subset of $\mathfrak{n}%
^{\ast}$ such that $\mathfrak{p}$ is an ideal subordinated to every linear
functional $\lambda\in\Omega.$ First, observe that $\Omega\cap[\mathfrak{p}%
,\mathfrak{p}]^{\ast}$ is open in $[\mathfrak{p},\mathfrak{p}]^{\ast}.$ Next,
for arbitrary $\ell\in\Omega\cap[\mathfrak{p},\mathfrak{p}]^{\ast},$ by
assumption $[\mathfrak{p},\mathfrak{p}]$ is contained in the kernel of $\ell.$
Thus, $[\mathfrak{p},\mathfrak{p}]$ must be a trivial vector space and it
follows that $\mathfrak{p}$ is commutative.
\end{proof}

The following result is established in Theorem $3.1.9,$ \cite{Corwin}

\begin{proposition}
A cross-section for the coadjoint orbits in $\Omega$ is
\begin{equation}
\Sigma=\left\{  \lambda\in\Omega:\lambda\left(  Z_{k}\right)  =0\text{ for all
}k\in\mathbf{e}\text{ }\right\}  \label{Sigma}%
\end{equation}

\end{proposition}

\subsection{Unitary Dual and Plancherel Theory}

The setting in which we are studying sampling spaces requires the following ingredients:

\begin{enumerate}
\item An explicit description of the irreducible representations occurring in
the decomposition of the left regular representation of $N.$

\item The Plancherel measure, and a formula for the Fourier (Plancherel)\ transform.

\item A description of left-invariant multiplicity-free spaces.
\end{enumerate}

\subsubsection{A Realization of the Irreducible Representations of $N$
\label{realization}}

The following discussion is mainly taken from Chapter $6,$ \cite{Folland}. Let
$G$ be a locally compact group, and let $K$ be a closed subgroup of $G.$ Let
us define $q:G\rightarrow G/K$ to be the canonical quotient map and let
$\varphi$ be a unitary representation of the group $K$ acting in some Hilbert
space which we call $\mathbf{H.}$ Next, let $\mathbf{K}_{1}$\textbf{ }be the
set of continuous $\mathbf{H}$-valued functions $f$ defined over $G$
satisfying the following properties:

\begin{itemize}
\item The image of the support of $f$ under the quotient map $q$ is compact.

\item $f\left(  gk\right)  =\left[  \varphi\left(  k\right)  ^{-1}f\right]
\left(  g\right)  $ for $g\in G$ and $k\in K.$
\end{itemize}

Clearly, $G$ acts on the set $\mathbf{K}_{1}$ by left translation. Now, to
simplify the presentation, let us suppose that $G/K$ admits a $G$-invariant
measure (this assumption is not always true.) However, since we are mainly
dealing with unimodular groups, the assumption holds. First, we endow
$\mathbf{K}_{1}$ with the following inner product: $\left\langle f,f^{\prime
}\right\rangle =\int_{G/K}\left\langle f\left(  g\right)  ,f^{\prime}\left(
g\right)  \right\rangle _{\mathbf{H}}\text{ }d\left(  gK\right)  \text{ for
}f,f^{\prime}\in\mathbf{K}_{1}.$ Second, let $\mathbf{K}$ be the Hilbert
completion of the space $\mathbf{K}_{1}$ with respect to this inner product.
The translation operators extend to unitary operators on $\mathbf{K}$ inducing
the unitary representation $\mathrm{Ind}_{K}^{G}\left(  \varphi\right)  $
which acts on $\mathbf{K}$ as follows:%
\[
\left[  \mathrm{Ind}_{K}^{G}\left(  \varphi\right)  \left(  x\right)
f\right]  \left(  g\right)  =f\left(  x^{-1}g\right)  \text{ for }%
f\in\mathbf{K.}%
\]
We notice that if $\varphi$ is a character, then the Hilbert space
$\mathbf{K}$ can be naturally identified with $L^{2}\left(  G/K\right)  .$ The
reader who is not familiar with these notions is invited to refer to Chapter
$6$ of the book of Folland \cite{Folland} for a thorough presentation.

For each linear functional in the set $\Sigma$ (see (\ref{Sigma})), there is a
corresponding unitary irreducible representation of $N$ which is realized as
acting in $L^{2}\left(  \mathbb{R}^{m}\right)  $ as follows. Define a
character $\chi_{\lambda}$ on the normal subgroup $\exp\left(  \mathfrak{p}%
(\lambda)\right)  $ such that $\chi_{\lambda}\left(  \exp X\right)  =e^{2\pi
i\lambda\left(  X\right)  }\text{ for }X\in\mathfrak{p}(\lambda).$ In order to
realize an irreducible representation corresponding to the linear functional
$\lambda$, induce the character $\chi_{\lambda}$ as follows:
\begin{equation}
\sigma_{\lambda}=\mathrm{Ind}_{P_{\lambda}}^{N}\left(  \chi_{\lambda}\right)
,\text{ where }P_{\lambda}=\exp\left(  \mathfrak{p}(\lambda)\right)  .
\label{irreducible}%
\end{equation}
The induced representation $\sigma_{\lambda}$ acts by left translations on the
Hilbert space
\begin{equation}%
\begin{array}
[c]{c}%
\mathbf{H}_{\lambda}=\left\{  f:N\xrightarrow{\hspace*{1cm}}\mathbb{C}%
:f\left(  xy\right)  =\chi_{\lambda}\left(  y\right)  ^{-1}f\left(  x\right)
\text{ for }y\in P_{\lambda}\right. \\
\left.  \text{and }\int_{N/P_{\lambda}}\left\vert f\left(  x\right)
\right\vert ^{2}d\left(  xP_{\lambda}\right)  <\infty\right\}  ,
\end{array}
\label{Hilbert}%
\end{equation}
which is endowed with the following inner product:
\[
\left\langle f,f^{\prime}\right\rangle =\int_{N/P_{\lambda}}f\left(  n\right)
\overline{f^{\prime}\left(  n\right)  }d\left(  nP_{\lambda}\right)  .
\]
Picking a cross-section in $N$ for $N/P_{\lambda},$ since $\chi_{\lambda}$ is
a character there is an obvious identification between $\mathbf{H}_{\lambda}$
and the Hilbert space $L^{2}\left(  N/P_{\lambda}\right)  =L^{2}\left(
\mathbb{R}^{m}\right)  .$

\subsubsection{The Plancherel Measure and the Plancherel
Transform\label{Plancherel}}

For a linear functional $\lambda\in\Omega,$ put $\mathbf{e}=\left\{
\mathbf{e}_{1}<\mathbf{e}_{2}<\cdots<\mathbf{e}_{2m}\right\}  $ and define
\begin{equation}
B\left(  \lambda\right)  =\left[  \lambda\left[  X_{\mathbf{e}_{i}%
},X_{\mathbf{e}_{j}}\right]  \right]  _{1\leq i,j\leq2m}. \label{B}%
\end{equation}
Then $B\left(  \lambda\right)  $ is a skew-symmetric invertible matrix of rank
$2m.$ Let $d\lambda$ be the Lebesgue measure on $\Sigma$ which is parametrized
by a Zariski subset of $\mathbb{R}^{n-2m}.$ Put
\[
d\mu\left(  \lambda\right)  =\left\vert \det B\left(  \lambda\right)
\right\vert ^{1/2}d\lambda.
\]
It is proved in Section $4.3$, \cite{Corwin} that up to multiplication by a
constant, the measure $d\mu\left(  \lambda\right)  $ is the Plancherel measure
for $N$. The group Fourier transform $\mathcal{F}$ is an operator-valued
bounded operator which is weakly defined on $L^{2}(N)\cap L^{1}(N)$ as
follows:
\begin{equation}
\sigma_{\lambda}\left(  f\right)  =\mathcal{F}\left(  f\right)  \left(
\lambda\right)  =\int_{\Sigma}f\left(  n\right)  \sigma_{\lambda}\left(
n^{-1}\right)  dn\text{ where }f\in L^{2}(N)\cap L^{1}(N). \label{Fourier}%
\end{equation}
Moreover, given $\mathbf{u},\mathbf{v}\in L^{2}\left(  \mathbb{R}^{m}\right)
,$ we have
\[
\left\langle \sigma_{\lambda}\left(  f\right)  \mathbf{u},\mathbf{v}%
\right\rangle =\int_{\Sigma}f\left(  n\right)  \left\langle \sigma_{\lambda
}\left(  n^{-1}\right)  \mathbf{u},\mathbf{v}\right\rangle dn.
\]
Next, the Plancherel transform is a unitary operator
\[
\mathcal{P}:L^{2}(N)\xrightarrow{\hspace*{1cm}}L^{2}\left(  \Sigma
,L^{2}\left(  \mathbb{R}^{m}\right)  \otimes L^{2}\left(  \mathbb{R}%
^{m}\right)  ,d\mu\left(  \lambda\right)  \right)
\]
which is obtained by extending the Fourier transform to $L^{2}(N)$. This
extension induces the equality
\[
\left\Vert f\right\Vert _{L^{2}\left(  N\right)  }^{2}=\int_{\Sigma}\left\Vert
\widehat{f}\left(  \sigma_{\lambda}\right)  \right\Vert _{\mathcal{HS}}%
^{2}d\mu\left(  \lambda\right)
\]
where $\widehat{f}\left(  \sigma_{\lambda}\right)  =\mathcal{P}f\left(
\lambda\right)  $. Let $L$ be the left regular representation of the nilpotent
group $N.$ It is easy to check that for almost every $\lambda\in\Sigma$ (with
respect to the Plancherel measure)
\[
\left(  \mathcal{P}L\left(  n\right)  \mathcal{P}^{-1}A\right)  \left(
\sigma_{\lambda}\right)  =\sigma_{\lambda}\left(  n\right)  \circ A\left(
\sigma_{\lambda}\right)  .
\]
In other words, the Plancherel transform intertwines the regular
representation with a direct integral of irreducible representations of $N.$
The irreducible representations occurring in the decomposition are
parametrized up to a null set by the manifold $\Sigma$ and each irreducible
representation occurs with infinite multiplicities in the decomposition.

\subsection{Bandlimited Multiplicity-Free Spaces}

Given any measurable set $\mathbf{A}\subseteq\Sigma$, it is easily checked
that the Hilbert space
\[
\mathcal{P}^{-1}\left(  L^{2}\left(  \mathbf{A,}\text{ }L^{2}\left(
\mathbb{R}^{m}\right)  \otimes L^{2}\left(  \mathbb{R}^{m}\right)  \right)
,d\mu\left(  \lambda\right)  \right)
\]
is a left-invariant subspace of $L^{2}\left(  N\right)  .$ Let us suppose that
$\mathbf{A}$ is a \textbf{bounded} subset of $\Sigma$ of positive Plancherel
measure. Letting $\left\vert \mathbf{A}\right\vert $ be the Lebesgue measure
of the set $\mathbf{A}$
\begin{equation}
\mu\left(  \mathbf{A}\right)  =\int_{\mathbf{A}}\left\vert \det B\left(
\lambda\right)  \right\vert ^{1/2}d\lambda\leq\left\vert \mathbf{A}\right\vert
\text{ }\sup\left\{  \left\vert \det B\left(  \lambda\right)  \right\vert
^{1/2}:\lambda\in\mathbf{A}\right\}  . \label{finite}%
\end{equation}
Next, since $\lambda\mapsto|\det B(\lambda)|^{\frac{1}{2}}$ is a continuous
function then $\mu(\mathbf{A})$ is finite. Fix a measurable field $\left(
\mathbf{u}\left(  \lambda\right)  \right)  _{\lambda\in\mathbf{A}}$ of unit
vectors in $L^{2}\left(  \mathbb{R}^{m}\right)  .$ Put%

\begin{equation}%
\begin{array}
[c]{c}%
\mathbf{H}_{\mathbf{A}}=\left\{  f\in L^{2}\left(  N\right)  :\mathcal{P}%
f\left(  \lambda\right)  =\left\{
\begin{array}
[c]{c}%
\mathbf{v}\left(  \lambda\right)  \otimes\mathbf{u}\left(  \lambda\right)
\text{ if }\lambda\in\mathbf{A}\\
0\otimes0\text{ if }\lambda\notin\mathbf{A}%
\end{array}
\right.  \right.  \text{ and }\\
\left.
\begin{array}
[c]{c}%
\left(  \mathbf{v}\left(  \lambda\right)  \otimes\mathbf{u}\left(
\lambda\right)  \right)  _{\lambda\in\mathbf{A}}\text{ is a measurable }\\
\text{field of rank-one operators}%
\end{array}
\text{ }\right\}  .
\end{array}
\label{HA}%
\end{equation}
Then $\mathbf{H}_{\mathbf{A}}$ is a left-invariant, band-limited and
multiplicity-free subspace of $L^{2}\left(  N\right)  .$ Let $h\in
\mathbf{H}_{\mathbf{A}}$ such that the Plancherel transform of $h$ is a
measurable field of rank-one operators. More precisely, let us assume that
\begin{equation}
\mathcal{P}h\left(  \sigma_{\lambda}\right)  =\widehat{h}\left(
\sigma_{\lambda}\right)  =\left\{
\begin{array}
[c]{c}%
\mathbf{u}\left(  \lambda\right)  \otimes\mathbf{u}\left(  \lambda\right)
\text{ if }\lambda\in\mathbf{A}\\
0\otimes0\text{ if }\lambda\notin\mathbf{A}%
\end{array}
\right.  \label{h}%
\end{equation}
and
\[
\left[  V_{h}^{L}\left(  f\right)  \right]  \left(  \exp\left(  X\right)
\right)  =\left\langle f,L\left(  \exp\left(  X\right)  \right)
h\right\rangle =f\ast h^{\ast}\left(  \exp\left(  X\right)  \right)
\]
where $h^{\ast}\left(  x\right)  =\overline{h\left(  x^{-1}\right)  }$ and
$f\ast g\left(  n\right)  =\int_{N}f\left(  m\right)  g\left(  m^{-1}n\right)
dm.$

\begin{proposition}
If $\mathbf{A}$ is a bounded subset of $\Sigma$ of positive Plancherel measure
and if $h$ is as given in (\ref{h})\ then $h$ is an admissible vector for the
representation $\left(  L,\mathbf{H}_{\mathbf{A}}\right)  .$
\end{proposition}

\begin{proof}
To check that $h$ is well-defined as an element of the Hilbert space
$\mathbf{H}_{\mathbf{A}},$ it is enough to verify that
\[
\left\Vert h\right\Vert _{L^{2}\left(  N\right)  }^{2}=\int_{\mathbf{A}%
}\left\Vert \mathbf{u}\left(  \lambda\right)  \otimes\mathbf{u}\left(
\lambda\right)  \right\Vert _{\mathcal{HS}}^{2}\text{ }d\mu\left(
\lambda\right)  =\mu\left(  \mathbf{A}\right)
\]
is finite. Next, for any vector $f\in\mathbf{H}_{\mathbf{A}},$ the square of
the norm of the image of $f$ under the map $V_{h}^{L}$ is computed as
follows:
\begin{align*}
\left\Vert V_{h}^{L}\left(  f\right)  \right\Vert _{L^{2}\left(  N\right)
}^{2}  &  =\int_{\mathbf{A}}\left\Vert \widehat{f}\left(  \sigma_{\lambda
}\right)  \left(  \mathbf{u}\left(  \lambda\right)  \otimes\mathbf{u}\left(
\lambda\right)  \right)  \right\Vert _{\mathcal{HS}}^{2}d\mu\left(
\lambda\right) \\
&  =\int_{\mathbf{A}}\left\langle \widehat{f}\left(  \sigma_{\lambda}\right)
\mathbf{u}\left(  \lambda\right)  ,\widehat{f}\left(  \sigma_{\lambda}\right)
\mathbf{u}\left(  \lambda\right)  \right\rangle _{L^{2}\left(  \mathbb{R}%
^{m}\right)  }d\mu\left(  \lambda\right)  .
\end{align*}
Letting $\widehat{f}\left(  \sigma_{\lambda}\right)  =\mathbf{v}\left(
\lambda\right)  \otimes\mathbf{u}\left(  \lambda\right)  $ where
$\mathbf{v}\left(  \lambda\right)  $ is in $L^{2}\left(  \mathbb{R}%
^{m}\right)  ,$ it follows that
\[
\left\Vert V_{h}^{L}\left(  f\right)  \right\Vert _{L^{2}\left(  N\right)
}^{2}=\int_{\mathbf{A}}\left\langle \mathbf{v}\left(  \lambda\right)
,\mathbf{v}\left(  \lambda\right)  \right\rangle _{L^{2}\left(  \mathbb{R}%
^{m}\right)  }\text{ }d\mu\left(  \lambda\right)  =\left\Vert f\right\Vert
_{L^{2}\left(  N\right)  }^{2}.
\]
In other words, the map $V_{h}^{L}$ defines an isometry from $\mathbf{H}%
_{\mathbf{A}}$ into $L^{2}\left(  N\right)  $ and the representation $\left(
L,\mathbf{H}_{\mathbf{A}}\right)  $ which is a subrepresentation of the left
regular representation of $N$ is admissible. Thus, the vector $h$ is an
admissible vector.
\end{proof}

It is also worth noting that $h$ is convolution idempotent in the sense that
$h=h\ast h^{\ast}=h^{\ast}.$ Next, $V_{h}^{L}\left(  \mathbf{H}_{\mathbf{A}%
}\right)  $ is a left-invariant vector subspace of $L^{2}\left(  N\right)  $
consisting of continuous functions. Moreover, the projection onto the Hilbert
space $V_{h}^{L}\left(  \mathbf{H}_{\mathbf{A}}\right)  $ is given by right
convolution in the sense that $V_{h}^{L}\left(  \mathbf{H}_{\mathbf{A}%
}\right)  =L^{2}\left(  N\right)  \ast h.$


\vskip 0.2 cm\noindent In order to simplify our presentation, we shall
naturally identify the Hilbert space $\mathbf{H}_{\mathbf{A}}$ with
$L^{2}\left(  \mathbf{A}\times\mathbb{R}^{m},d\mu\left(  \lambda\right)
dt\right)  $. This identification is given by the map
\[
\left(  \mathbf{v}\left(  \lambda\right)  \otimes\mathbf{u}\left(
\lambda\right)  \right)  _{\lambda\in\mathbf{A}}\mathbf{\mapsto}\left[
\mathbf{v}\left(  \lambda\right)  \right]  \left(  t\right)  :=\mathbf{v}%
\left(  \lambda,t\right)
\]
for any measurable field of rank-one operators $\left(  \mathbf{v}\left(
\lambda\right)  \otimes\mathbf{u}\left(  \lambda\right)  \right)  _{\lambda
\in\mathbf{A}}.$


\begin{lemma}
\label{admissible and sampling}Let $\pi$ be a unitary representation of a
group $N$ acting in a Hilbert space $\mathbf{H}_{\pi}.$ Assume that $\pi$ is
admissible, and let $h$ be an admissible vector for $\pi.$ Furthermore,
suppose that $\pi\left(  \Gamma\right)  h$ is a tight frame with frame bound
$C_{h}$. Then the vector space $V_{h}\left(  \mathbf{H}_{\pi}\right)  $ is a
left-invariant closed subspace of $L^{2}\left(  N\right)  $ consisting of
continuous functions and $V_{h}\left(  \mathbf{H}_{\pi}\right)  $ is a
sampling space with sinc-type function $\frac{1}{C_{h}}V_{h}\left(  h\right)
.$
\end{lemma}

Lemma \ref{admissible and sampling} is proved in Proposition $2.54$,
\cite{Fuhr cont}. This result establishes a connection between admissibility
and sampling theories. This connection will play a central role in the proof
of our main results. The following result is a slight extension of Proposition
$2.61$ \cite{Fuhr cont}, the proof given here is essentially inspired by the
one given in the Monograph \cite{Fuhr cont}.

\begin{lemma}
\label{sampling}Let $\Gamma$ be a discrete subgroup of $N$ with positive
co-volume Let $\pi$ be a unitary representation of $N$ acting in a Hilbert
space $\mathbf{H}_{\pi}.$ If the restriction of $\pi$ to the discrete subgroup
$\Gamma$ is unitarily equivalent to a subrepresentation of the left regular
representation of $\Gamma$ then there exists a subspace of $L^{2}\left(
N\right)  $ which is a sampling space with respect to $\Gamma.$ Moreover, if
$\pi$ is equivalent to the left regular representation of $\Gamma$ then there
exists a subspace of $L^{2}\left(  N\right)  $ which is a sampling space with
the interpolation property with respect to $\Gamma.$
\end{lemma}

\begin{proof}
Let $T:\mathbf{H}_{\pi}\rightarrow\mathbf{H}\subset l^{2}\left(
\Gamma\right)  $ be a unitary map which is intertwining the restricted
representation of $\pi$ to $\Gamma$ with a representation which is a
subrepresentation of the left regular representation of the lattice $\Gamma.$
Since $\Gamma$ is a discrete group, the left regular representation of
$\Gamma$ is admissible. To see this, let $\kappa$ be the sequence which is
equal to one at the identity of $\Gamma$ and zero everywhere else. By shifting
the sequence $\kappa$ by elements in $\Gamma,$ we generate an orthonormal
basis for the Hilbert space $l^{2}\left(  \Gamma\right)  .$ Now, let
$P:l^{2}\left(  \Gamma\right)  \rightarrow\mathbf{H}$ be an orthogonal
projection. Next, the vector $\eta=T^{-1}\left(  P\left(  \kappa\right)
\right)  $ is an admissible vector for $\pi|_{\Gamma}$ as well. We recall that
$V_{\eta}^{\pi}\left(  f\right)  =\left\langle f,\pi\left(  \cdot\right)
\eta\right\rangle .$ Let $N=A\Gamma$ where $A$ is a set of finite measure with
respect to the Haar measure of $N.$ Without loss of generality, let us assume
that a Haar measure for $N$ is fixed so that $\left\vert A\right\vert =1.$
Then
\begin{align*}
\left\Vert V_{\eta}^{\pi}\left(  f\right)  \right\Vert _{L^{2}\left(
N\right)  }^{2}  &  =\int_{N}\left\vert \left\langle f,\pi\left(  x\right)
\eta\right\rangle \right\vert ^{2}dx\\
&  =\int_{A}{\sum\limits_{\gamma\in\Gamma}}\left\vert \left\langle
f,\pi\left(  x\gamma\right)  \eta\right\rangle \right\vert ^{2}dx\\
&  =\int_{A}{\sum\limits_{\gamma\in\Gamma}}\left\vert \left\langle \pi\left(
x^{-1}\right)  f,\pi\left(  \gamma\right)  \eta\right\rangle \right\vert
^{2}dx.
\end{align*}
Next, since $\pi\left(  \Gamma\right)  \eta$ is a Parseval frame in
$\mathbf{H}_{\pi},$
\[
{\sum\limits_{\gamma\in\Gamma}}\left\vert \left\langle \pi\left(
x^{-1}\right)  f,\pi\left(  \gamma\right)  \eta\right\rangle \right\vert
^{2}=\left\Vert \pi\left(  x^{-1}\right)  f\right\Vert _{\mathbf{H}_{\pi}}^{2}%
\]
and it follows that%
\[
\left\Vert V_{\eta}^{\pi}\left(  f\right)  \right\Vert _{L^{2}\left(
N\right)  }^{2}=\int_{A}\left\Vert \pi\left(  x^{-1}\right)  f\right\Vert
_{\mathbf{H}_{\pi}}^{2}dx=\left\Vert f\right\Vert _{\mathbf{H}_{\pi}}%
^{2}\left\vert A\right\vert =\left\Vert f\right\Vert _{\mathbf{H}_{\pi}}^{2}.
\]
Thus $\eta$ is a continuous wavelet for the representation $\pi$, $\pi\left(
\Gamma\right)  \left(  \eta\right)  $ is a Parseval frame, and the Hilbert
space $V_{\eta}^{\pi}\left(  \mathbf{H}_{\pi}\right)  $ is a sampling space of
$L^{2}\left(  N\right)  $ with respect to the lattice $\Gamma.$ Now, for the
second part, if we assume that $\pi$ is equivalent to the left regular
representation, then the operator $P$ described above is just the identity
map. Next, $\pi\left(  \Gamma\right)  \eta=\pi\left(  \Gamma\right)  \left(
T^{-1}\left(  \kappa\right)  \right)  $ is an orthonormal basis of
$\mathbf{H}_{\pi}$ and $V_{\eta}\left(  \eta\right)  $ is a sinc-type
function. It follows from Theorem $2.56$ \cite{Fuhr cont} that $V_{\eta
}\left(  \mathbf{H}_{\pi}\right)  $ is a sampling space of $L^{2}\left(
N\right)  $ which has the interpolation property with respect to the lattice
$\Gamma.$
\end{proof}

\begin{remark}
Let $\mathbf{H}_{\mathbf{A}}$ be the Hilbert space of band-limited functions
as described in (\ref{HA}). We recall that $\Gamma=\exp\left(  \mathbb{Z}%
X_{1}\right)  \cdots\exp\left(  \mathbb{Z}X_{n}\right)  $ is a discrete
uniform subgroup of $N.$ Since $\mathbf{H}_{\mathbf{A}}$ is left-invariant,
the regular representation of $N$ admits a subrepresentation obtained by
restricting the action of the left regular representation to the Hilbert space
$\mathbf{H}_{\mathbf{A}}.$ Let us denote such a representation by
$L_{\mathbf{H}_{\mathbf{A}}}.$ Furthermore, let $L_{\mathbf{H}_{\mathbf{A}%
},\Gamma}$ be the restriction of $L_{\mathbf{H}_{\mathbf{A}}}$ to $\Gamma.$ If
the representation $L_{\mathbf{H}_{\mathbf{A}},\Gamma}$ is unitarily
equivalent to a subrepresentation of the left regular representation of the
discrete group $\Gamma$ then according to arguments used in the proof of Lemma
\ref{sampling}, there exists a vector $\eta$ such that $V_{\eta}^{L}\left(
\mathbf{H}_{\mathbf{A}}\right)  $ is a sampling space of $L^{2}\left(
N\right)  $ with respect to the discrete uniform group $\Gamma.$ In the
present work, we are aiming to find conditions on the spectral set
$\mathbf{A}$ which guarantees that $L_{\mathbf{H}_{\mathbf{A}},\Gamma}$ is
unitarily equivalent to a subrepresentation of the left regular representation
of the discrete group $\Gamma.$ We shall also prove that under the assumptions
given in Condition \ref{cond}, it is possible to find $\mathbf{A}$ such that
$V_{\eta}^{L}\left(  \mathbf{H}_{\mathbf{A}}\right)  $ is a sampling space of
$L^{2}\left(  N\right)  $ with respect to the discrete uniform group $\Gamma.$
\end{remark}

\section{Intermediate Results\label{intermediateresult}}

Let us now fix assumptions and specialize the theory of harmonic analysis of
nilpotent Lie groups to the class of groups being considered here.

\vskip0.3cm \noindent Let $N$ be a simply connected, connected non-commutative
nilpotent Lie group with Lie algebra $\mathfrak{n}$ with rational structure
constants such that $N=PM=\exp\left(  \mathfrak{p}\right)  \exp\left(
\mathfrak{m}\right)  $ where $\mathfrak{p}$ and $\mathfrak{m}$ are commutative
Lie algebras, and $\mathfrak{p}$ is an ideal of $\mathfrak{n}.$ We fix a
strong Malcev basis
\begin{equation}
\left\{  Z_{1},\cdots,Z_{p},A_{1},\cdots,A_{m}\right\}  \label{Malcev}%
\end{equation}
for $\mathfrak{n}$ such that $\left\{  Z_{1},\cdots,Z_{p}\right\}  $ is a
basis for $\mathfrak{p}$ and $\left\{  A_{1},\cdots,A_{m}\right\}  $ is a
basis for $\mathfrak{m.}$ Therefore, $N$ is isomorphic to the semi-direct
product group $P\rtimes M$ endowed with the multiplication law
\[
\left(  \exp Z,\exp A\right)  \left(  \exp Z^{\prime},\exp A^{\prime}\right)
=\left(  \exp\left(  Z+e^{adA}Z^{\prime}\right)  ,\exp\left(  A+A^{\prime
}\right)  \right)  .
\]
Moreover it is assumed that
\[
\Gamma=\exp\left(  {\sum\limits_{k=1}^{p}}\mathbb{Z}Z_{k}\right)  \exp\left(
{\sum\limits_{k=1}^{m}}\mathbb{Z}A_{k}\right)
\]
is a discrete uniform subgroup of $N.$ Indeed, in order to ensure that
$\Gamma$ is a discrete uniform group, it is enough to pick $\left\{
A_{1},\cdots,A_{m}\right\}  $ such that the matrix representation of
$e^{adA_{k}}|\mathfrak{p}$ with respect to the basis $\left\{  Z_{1}%
,\cdots,Z_{p}\right\}  $ has entries in $\mathbb{Z}.$

\vskip 0.3cm \noindent If $\mathbf{M}\left(  \lambda\right)  $ is the
skew-symmetric matrix described in (\ref{Mlambda}) then
\[
\mathbf{M}\left(  \lambda\right)  =\left[
\begin{array}
[c]{cccccc}%
0 & \cdots & 0 & \lambda\left[  Z_{1},A_{1}\right]  & \cdots & \lambda\left[
Z_{1},A_{m}\right] \\
\vdots & \ddots & \vdots & \vdots & \ddots & \vdots\\
0 & \cdots & 0 & \lambda\left[  Z_{p},A_{1}\right]  & \cdots & \lambda\left[
Z_{p},A_{m}\right] \\
\lambda\left[  A_{1},Z_{1}\right]  & \cdots & \lambda\left[  A_{1}%
,Z_{p}\right]  & 0 & \cdots & 0\\
\vdots & \ddots & \vdots & \vdots & \ddots & \vdots\\
\lambda\left[  A_{m},Z_{1}\right]  & \cdots & \lambda\left[  A_{m}%
,Z_{p}\right]  & 0 & \cdots & 0
\end{array}
\right]  .
\]

\vskip0.3cm \noindent Regarding $\mathbf{M}\left(  \lambda\right)  $ as the
matrix representation of a linear operator acting on $\mathfrak{n,}$ we recall
that the null-space of $\mathbf{M}\left(  \lambda\right)  $ corresponds to
$\mathfrak{n}\left(  \lambda\right)  .$ Now, let $\mathbf{M}_{k}\left(
\lambda\right)  $ be the matrix obtained by retaining the first $k$ columns of
$\mathbf{M}\left(  \lambda\right)  $ (see illustration below)
\[
\underset{\mathbf{M}\left(  \lambda\right)  =\mathbf{M}_{n}\left(
\lambda\right)  }{\mathbf{M}\left(  \lambda\right)  =\underbrace
{\underset{\mathbf{M}_{p+1}\left(  \lambda\right)  }{\underbrace
{\underset{\mathbf{M}_{p}\left(  \lambda\right)  }{\underbrace{\underset
{\mathbf{M}_{1}\left(  \lambda\right)  }{\underbrace{\left[
\begin{array}
[c]{c}%
0\\
\vdots\\
0\\
\lambda\left[  A_{1},Z_{1}\right] \\
\vdots\\
\lambda\left[  A_{m},Z_{1}\right]
\end{array}
\right.  }}%
\begin{array}
[c]{cc}%
\cdots & 0\\
\ddots & \vdots\\
\cdots & 0\\
\cdots & \lambda\left[  A_{1},Z_{p}\right] \\
\ddots & \vdots\\
\cdots & \lambda\left[  A_{m},Z_{p}\right]
\end{array}
}}%
\begin{array}
[c]{c}%
\lambda\left[  Z_{1},A_{1}\right] \\
\vdots\\
\lambda\left[  Z_{p},A_{1}\right] \\
0\\
\vdots\\
0
\end{array}
}}\left.
\begin{array}
[c]{cc}%
\cdots & \lambda\left[  Z_{1},A_{m}\right] \\
\ddots & \vdots\\
\cdots & \lambda\left[  Z_{p},A_{m}\right] \\
\cdots & 0\\
\ddots & \vdots\\
\cdots & 0
\end{array}
\right]  }}.
\]
Although $\mathbf{M}_{0}\left(  \lambda\right)  $ is not defined, we shall
need to assume that $\mathrm{rank}\left(  \mathbf{M}_{0}\left(  \lambda
\right)  \right)  =0.$ Put
\[
X_{1}=Z_{1},\cdots,X_{p}=Z_{p},X_{p+1}=A_{1},\cdots,X_{n}=A_{m}.
\]

\begin{lemma}
Given $\lambda\in\mathfrak{n}^{\ast},$ the following holds true.
\[
\left\{  1\leq k\leq n:\mathrm{rank}\left(  \mathbf{M}_{k}\left(
\lambda\right)  \right)  >\mathrm{rank}\left(  \mathbf{M}_{k-1}\left(
\lambda\right)  \right)  \right\}  =\left\{  1\leq k\leq n:\mathfrak{n}%
_{k}\text{ }\nsubseteq\text{ }\mathfrak{n}_{k-1}+\mathfrak{n}\left(
\lambda\right)  \right\}  .
\]

\end{lemma}

\begin{proof}
First, assume that the rank of $\mathbf{M}_{i}(\lambda)$ is greater than the
rank of $\mathbf{M}_{i-1}(\lambda).$ Then it is clear that $X_{i}$ cannot be
in the null-space of the matrix $\mathbf{M}(\lambda).$ Thus, $\mathfrak{n}%
_{i}=\mathfrak{n}_{i-1}+\mathbb{R}X_{i}\nsubseteq\mathfrak{n}_{i-1}%
+\mathfrak{n}(\lambda).$ Next, if $\mathfrak{n}_{i}\nsubseteq\mathfrak{n}%
_{i-1}+\mathfrak{n}(\lambda)$ and $\mathfrak{n}_{i}=\mathfrak{n}%
_{i-1}+\mathbb{R}X_{i}$ since the basis element $X_{i}$ cannot be in
$\mathfrak{n}(\lambda).$ Thus the rank of $\mathbf{M}_{i}(\lambda)$ is greater
than the rank of $\mathbf{M}_{i-1}(\lambda)$ and the stated result is established.
\end{proof}

It is proved in Theorem $3.1.9,$ \cite{Corwin} that there exist a Zariski open
subset $\Omega$ of $\mathfrak{n}^{\ast}$ and a fixed set $\mathbf{e}%
\subset\left\{  1,2,\cdots,n\right\}  $ such that the map
\[
\lambda\mapsto\left\{  1\leq k\leq n:\mathrm{rank}\left(  \mathbf{M}%
_{k}\left(  \lambda\right)  \right)  >\mathrm{rank}\left(  \mathbf{M}%
_{k-1}\left(  \lambda\right)  \right)  \right\}  =\mathbf{e}%
\]
is constant, $\Omega$ is invariant under the coadjoint action of $N$ and
\[
\Sigma=\left\{  \lambda\in\Omega:\lambda\left(  X_{k}\right)  =0\text{ for all
}k\in\mathbf{e}\text{ }\right\}
\]
is an algebraic set which is a cross-section for the coadjoint orbits of $N$
in $\Omega,$ as well as a parameterizing set for the unitary dual of $N.$

\begin{lemma}
If $\mathfrak{p}$ is a constant polarization for $\mathfrak{n}$ then the set
\[
\left\{  p+1,p+2,\cdots,p+m=n\right\}
\]
is contained in $\mathbf{e}$ and $\mathrm{card}\left(  \mathbf{e}\right)
=2m.$
\end{lemma}

\begin{proof}
Let $\lambda\in\Sigma.$ Let us suppose that there exists $k\in\left\{
p+1,p+2,\cdots,p+m=n\right\}  $ such that $k$ is not an element of the set
$\mathbf{e}\left(  \lambda\right)  =\mathbf{e.}$ Without loss of generality,
since the algebra generated by the $A_{j}$ is commutative, we may assume that
$X_{k}=A_{1}.$ Indeed for any permutation $\sigma\in S_{m},$ $\left\{
Z_{1},\cdots,Z_{p},A_{\sigma(1)},\cdots A_{\sigma(m)}\right\}  $ is a Malcev
basis for $\mathfrak{n}.$ Since $\mathrm{rank}\left(  \mathbf{M}_{k}\left(
\lambda\right)  \right)  =\mathrm{rank}\left(  \mathbf{M}_{k-1}\left(
\lambda\right)  \right)  $ are since the matrices $M_{k-1}(\lambda),$ and
$M_{k}(\lambda)$ are given as shown below
\[
\overset{=\mathbf{M}_{k-1}\left(  \lambda\right)  }{\mathbf{M}_{k}\left(
\lambda\right)  =\overbrace{\left[
\begin{array}
[c]{ccc}%
0 & \cdots & 0\\
\vdots & \ddots & \vdots\\
0 & \cdots & 0\\
\lambda\left[  A_{1},Z_{1}\right]  & \cdots & \lambda\left[  A_{1}%
,Z_{p}\right] \\
\vdots & \ddots & \vdots\\
\lambda\left[  A_{m},Z_{1}\right]  & \cdots & \lambda\left[  A_{m}%
,Z_{p}\right]
\end{array}
\right.  }}\left.
\begin{array}
[c]{c}%
\lambda\left[  Z_{1},A_{1}\right] \\
\vdots\\
\lambda\left[  Z_{p},A_{1}\right] \\
0\\
\vdots\\
0
\end{array}
\right]  ,
\]
it is clear that for all $\lambda\in\Sigma,$ the last column of the matrix
$\mathbf{M}_{k}\left(  \lambda\right)  $ is equal to zero. It follows that
$\mathfrak{p}+\mathbb{R}X_{k}$ is a commutative algebra and $\left[
\mathfrak{p}+\mathbb{R}X_{k},\mathfrak{p}+\mathbb{R}X_{k}\right]  $ is
contained in the kernel of the linear functional $\lambda\in\Sigma;$
contradicting the fact that $\mathfrak{p}$ is a maximal algebra satisfying the
condition that $\left[  \mathfrak{p,p}\right]  \subseteq\ker\lambda.$ The
second part of the lemma is true because $\mathbf{M}(\lambda)$ is a skew
symmetric rank of positive rank.
\end{proof}

\begin{remark}
From now on, we shall assume that $\mathfrak{p}$ is a constant polarization
ideal for $\mathfrak{n}.$ Since $P=\exp\mathfrak{p}$ is normal in $N,$ it is
clear that the dual of the Lie algebra $\mathfrak{p}$ is invariant under the
coadjoint action of the commutative group $M$.
\end{remark}

\begin{remark}
\label{action of P}\text{ Let} $\beta:\Sigma\times\mathbb{R}^{m}%
\rightarrow\mathbb{R}^{p}$ be the mapping defined by
\[
\beta\left(  \lambda,t_{1},\cdots,t_{m}\right)  =\exp\left(  t_{1}A_{1}%
+\cdots+t_{m}A_{m}\right)  \cdot\lambda|\mathfrak{p}^{\ast}%
\]
where $\cdot$ stands for the coadjoint action. In vector-form, $\beta\left(
\lambda,t_{1},\cdots,t_{m}\right)  $ is easily computed as follows. Let
$\mathfrak{P}\left(  A(t)\right)  $ be the transpose of the matrix
representation of $e^{\left.  -ad\left(  \sum_{k=1}^{m}t_{k}A_{k}\right)
\right\vert \mathfrak{p}}$ with respect to the ordered basis $\left\{
Z_{k}:1\leq k\leq p\right\}  .$ We write
\[
\mathfrak{P}\left(  A(t)\right)  =\left[  e^{\left.  -ad\left(  \sum_{k=1}%
^{m}t_{k}A_{k}\right)  \right\vert \mathfrak{p}}\right]  ^{T}%
\]
and
\[
\beta\left(  \lambda,t_{1},\cdots,t_{m}\right)  \equiv\mathfrak{P}\left(
A(t)\right)  \left[
\begin{array}
[c]{c}%
f_{1}\\
\vdots\\
f_{p}%
\end{array}
\right]  \text{ where }\lambda=\sum_{k=1}^{p}f_{k}Z_{k}^{\ast}%
\]
$\left\{  Z_{k}^{\ast}:1\leq k\leq p\right\}  $ is a dual basis to $\left\{
Z_{k}:1\leq k\leq p\right\}  .$ We shall generally make no distinction between
linear functionals and their representations as either row or column vectors.
Secondly for any linear functional $\lambda=\sum_{k=1}^{p}f_{k}Z_{k}^{\ast}%
\in\Sigma$ since $\mathfrak{P}\left(  A(t)\right)  $ is a unipotent matrix,
the components of $\beta\left(  \lambda,t_{1},\cdots,t_{m}\right)  $ are
polynomials in the variables $f_{k}$ where $k\notin\mathbf{e}$ and
$t_{1},\cdots t_{m}.$ 
\end{remark}

\begin{example}
Let us suppose that $\mathfrak{p}$ is spanned with $Z_{1},Z_{2},Z_{3}$ and
$\mathfrak{m}$ is spanned by $A_{1},A_{2}$ such that
\[
\left[  A_{1},Z_{3}\right]  =Z_{1},\left[  A_{2},Z_{2}\right]  =Z_{1},\left[
A_{2},Z_{3}\right]  =Z_{2}.
\]
Then
\[
\left[  ad\left(  t_{1}A_{1}+t_{2}A_{2}\right)  |\mathfrak{p}\right]  =\left[
\begin{array}
[c]{ccc}%
0 & t_{2} & t_{1}\\
0 & 0 & t_{2}\\
0 & 0 & 0
\end{array}
\right]  \text{ and }\mathfrak{P}\left(  A(t)\right)  =\left[
\begin{array}
[c]{ccc}%
1 & 0 & 0\\
-t_{2} & 1 & 0\\
\frac{1}{2}t_{2}^{2}-t_{1} & -t_{2} & 1
\end{array}
\right]  .
\]

\end{example}

\begin{lemma}
\label{equivalency}Let $\lambda\in\Sigma.$ $\mathfrak{p}$ is a polarizing
ideal subordinated to the linear functional $\lambda$ if and only if for any
given $X\in\mathfrak{n}$, $\mathfrak{p}$ is also a polarizing ideal
subordinated to the linear functional $\exp\left(  X\right)  \cdot\lambda.$
\end{lemma}

\begin{proof}
Appealing to Proposition $1.3.6$ in \cite{Corwin}, $\mathfrak{p}$ is a
polarizing algebra subordinated to $\lambda$ if and only if $e^{adX}%
\mathfrak{p}=\mathfrak{p}$ is a polarizing algebra subordinated to
$\exp\left(  X\right)  \cdot\lambda.$
\end{proof}

\begin{lemma}
\label{beta 1}If for each $\lambda\in\Sigma,$ $\mathfrak{p}=\mathbb{R}$-span
$\left\{  Z_{1},\cdots,Z_{p}\right\}  $ is a commutative polarizing ideal
which is subordinated to the linear functional $\lambda$ then $\beta$ defines
a diffeomorphism between $\Sigma\times\mathbb{R}^{m}$ and its range.
\end{lemma}

\begin{proof}
In order to prove this result, it is enough to show that $\beta$ is a
bijective smooth map with constant full rank (see Theorem $6.5,$ \cite{Lee}).
In order to establish this fact, we will need to derive a precise formula for
the coadjoint action. Let us define $\beta^{\mathbf{e}}:\Sigma\times
\mathbb{R}^{2m}$ such that
\[
\beta^{\mathbf{e}}\left(  \lambda,t_{\mathbf{e}_{1}},\cdots,t_{\mathbf{e}%
_{2m}}\right)  =\exp\left(  t_{\mathbf{e}_{1}}X_{\mathbf{e}_{1}}\right)
\cdots\exp\left(  t_{\mathbf{e}_{2m}}X_{\mathbf{e}_{2m}}\right)  \cdot\lambda
\]
and $\left\{  \mathbf{e}_{1}<\mathbf{e}_{2}<\cdots<\mathbf{e}_{2m}\right\}
=\mathbf{e.}$ For a fixed linear functional $\lambda$ in the cross-section
$\Sigma,$ the map
\[
\left(  t_{\mathbf{e}_{1}},\cdots,t_{\mathbf{e}_{2m}}\right)  \mapsto
\exp\left(  t_{\mathbf{e}_{1}}X_{\mathbf{e}_{1}}\right)  \cdots\exp\left(
t_{\mathbf{e}_{2m}}X_{\mathbf{e}_{2m}}\right)  \cdot\lambda
\]
defines a diffeomorphism between $\mathbb{R}^{2m}$ and the $N$-orbit of
$\lambda$ which is a closed submanifold of the dual of the Lie algebra
$\mathfrak{n.}$ Next, since all orbits in $\Omega$ are $2m$-dimensional
manifolds, the map $\beta^{\mathbf{e}}$ is a bijection with constant
full-rank. Thus, $\beta^{\mathbf{e}}$ defines a diffeomorphism between
$\Sigma\times\mathbb{R}^{2m}$ and $\Omega$ . Next, there exist indices
$i_{1},\cdots,i_{m}\leq p$ such that for $g=\exp\left(  t_{\mathbf{e}_{1}%
}X_{\mathbf{e}_{1}}\right)  \cdots\exp\left(  t_{\mathbf{e}_{2m}}%
X_{\mathbf{e}_{2m}}\right)  $ we have
\[
g\cdot\lambda=\exp\left(  \sum_{k=1}^{m}t_{i_{k}}Z_{i_{k}}\right)  \exp\left(
\sum_{j=1}^{m}s_{j}A_{j}\right)  \cdot\lambda.
\]
In order to compute the coadjoint action of $N$ on the linear functional
$\lambda,$ it is quite convenient to identify $\mathfrak{n}^{\ast}$ with
$\mathfrak{p}^{\ast}\times\mathfrak{m}^{\ast}$ via the map
\[
\iota:\mathfrak{n}^{\ast}=\mathfrak{p}^{\ast}+\mathfrak{m}^{\ast}%
\rightarrow\mathfrak{p}^{\ast}\times\mathfrak{m}^{\ast}%
\]
which is defined as follows:
\[
\iota\left(  f_{1}+f_{2}\right)  =\left[
\begin{array}
[c]{c}%
f_{1}\\
f_{2}%
\end{array}
\right]  \text{ where }f_{1}\in\mathfrak{p}^{\ast}\text{ and }f_{2}%
\in\mathfrak{m}^{\ast}\mathfrak{.}%
\]
Thus, for any linear functional $\lambda\in\Sigma,$
\[
\iota\left(  \lambda\right)  =\left[
\begin{array}
[c]{c}%
f\\
0
\end{array}
\right]
\]
for some $f\in\mathfrak{p}.$ Put
\[
A\left(  s\right)  =\sum_{j=1}^{m}\left(  s_{j}A_{j}\right)  \in
\mathfrak{m}\text{, }Z\left(  t\right)  =\sum_{k=1}^{m}\left(  t_{i_{k}%
}Z_{i_{k}}\right)  \in\mathfrak{p}%
\]
where $s=\left(  s_{1},\cdots,s_{m}\right)  ,t=\left(  t_{i_{1}}%
,\cdots,t_{i_{m}}\right)  ,$ and let $\left[  e^{-adA\left(  s\right)
|\mathfrak{p}}\right]  $ be the matrix representation of the linear map
$e^{-adA\left(  s\right)  |\mathfrak{p}}$ which is obtained by exponentiating
$-adA\left(  s\right)  $ restricted to the vector space $\mathfrak{p}.$
Clearly, with the fixed choice of the Malcev basis described in (\ref{Malcev}%
), it is easy to check that (see Remark \ref{action of P})
\begin{equation}
\iota\left(  \exp\left(  A\left(  s\right)  \right)  \cdot\lambda\right)
=\left[
\begin{array}
[c]{c}%
\mathfrak{P}\left(  A(s)\right)  f\\
0
\end{array}
\right]  \text{ }\label{comp1}%
\end{equation}
and
\begin{equation}
\iota\left(  \exp\left(  Z\left(  t\right)  \right)  \cdot\lambda\right)
=\left[
\begin{array}
[c]{c}%
f\\
\sigma\left(  t,f\right)
\end{array}
\right]  \label{comp2}%
\end{equation}
where $\left(  t_{i_{1}},\cdots,t_{i_{m}}\right)  \mapsto\sigma\left(
t_{i_{1}},\cdots,t_{i_{m}},f\right)  $ is an $m\times1$ vector-valued
function. Putting (\ref{comp1}) and (\ref{comp2}) together,
\[
\exp\left(  Z\left(  t\right)  \right)  \exp\left(  A\left(  s\right)
\right)  \cdot\lambda=\left[
\begin{array}
[c]{c}%
\mathfrak{P}\left(  A(s)\right)  f\\
\sigma\left(  t,\mathfrak{P}\left(  A(s)\right)  f\right)
\end{array}
\right]  .
\]
In order to compute the Jacobian of the map
\begin{equation}
\left(  f,s_{1},\cdots,s_{m},t_{i_{1}},\cdots,t_{i_{m}}\right)  \mapsto\left[
\begin{array}
[c]{c}%
\mathfrak{P}\left(  A(s)\right)  f\\
\sigma\left(  t,\mathfrak{P}\left(  A(s)\right)  f\right)
\end{array}
\right]  \label{betaprime}%
\end{equation}
at the point $\left(  f,s,t\right)  ,$ we set
\[
\left[
\begin{array}
[c]{c}%
\mathfrak{P}\left(  A(s)\right)  f\\
\sigma\left(  t,\mathfrak{P}\left(  A(s)\right)  f\right)
\end{array}
\right]  =\left[
\begin{array}
[c]{c}%
\beta_{1}^{\mathbf{e}}\left(  f,s\right)  \\
\vdots\\
\beta_{p}^{\mathbf{e}}\left(  f,s\right)  \\
\beta_{p+1}^{\mathbf{e}}\left(  f,s,t\right)  \\
\vdots\\
\beta_{n}^{\mathbf{e}}\left(  f,s,t\right)
\end{array}
\right]
\]
where
\[
\mathfrak{P}\left(  A(s)\right)  f=\left[
\begin{array}
[c]{c}%
\beta_{1}^{\mathbf{e}}\left(  f,s\right)  \\
\vdots\\
\beta_{p}^{\mathbf{e}}\left(  f,s\right)
\end{array}
\right]  \text{ and }\sigma\left(  t,\mathfrak{P}\left(  A(s)\right)
f\right)  =\left[
\begin{array}
[c]{c}%
\beta_{p+1}^{\mathbf{e}}\left(  f,s,t\right)  \\
\vdots\\
\beta_{n}^{\mathbf{e}}\left(  f,s,t\right)
\end{array}
\right]  .
\]
Now, let $\psi:\Sigma\times\mathbb{R}^{2m}\rightarrow\mathbb{R}^{\dim\Sigma
}\times\mathbb{R}^{2m}$ such that
\[
\left(  \lambda,z\right)  \mapsto\psi\left(  \left(  \lambda,z\right)
\right)  =\left(  \left(  \ell_{1},\cdots,\ell_{\dim\Sigma}\right)  ,z\right)
\text{ and }\ell=\left(  \ell_{1},\cdots,\ell_{\dim\Sigma}\right)  .
\]
Then the pair $\left(  \Sigma\times\mathbb{R}^{2m},\psi\right)  $ is a smooth
chart around the point $\left(  \lambda,z\right)  .$ Computing the Jacobian of
(\ref{betaprime}) in local coordinates, we obtain the following matrix%
\begin{equation}%
\begin{array}
[c]{c}%
\left[
\begin{array}
[c]{c}%
A\left(  \ell,s\right)  \\
C\left(  \ell,s,t\right)
\end{array}%
\begin{array}
[c]{c}%
\\
D\left(  \ell,s,t\right)
\end{array}
\right]
\end{array}
\label{BigJac}%
\end{equation}
where $\left(  \ell,s\right)  \mapsto A\left(  \ell,s\right)  $ is a
matrix-valued function of order $p$ given by
\[
\left[
\begin{array}
[c]{cccccc}%
\dfrac{\partial\left(  \beta_{1}^{\mathbf{e}}\psi^{-1}\left(  \ell,s\right)
\right)  }{\partial\ell_{1}} & \cdots & \dfrac{\partial\left(  \beta
_{1}^{\mathbf{e}}\psi^{-1}\left(  \ell,s\right)  \right)  }{\partial\ell
_{\dim\Sigma}} & \dfrac{\partial\left(  \beta_{1}^{\mathbf{e}}\psi^{-1}\left(
\ell,s\right)  \right)  }{\partial s_{1}} & \cdots & \dfrac{\partial\left(
\beta_{1}^{\mathbf{e}}\psi^{-1}\left(  \ell,s\right)  \right)  }{\partial
s_{m}}\\
\vdots & \ddots & \vdots & \vdots & \ddots & \vdots\\
\dfrac{\partial\left(  \beta_{p}^{\mathbf{e}}\psi^{-1}\left(  \ell,s\right)
\right)  }{\partial\ell_{1}} & \cdots & \dfrac{\partial\left(  \beta
_{p}^{\mathbf{e}}\psi^{-1}\left(  \ell,s\right)  \right)  }{\partial\ell
_{\dim\Sigma}} & \dfrac{\partial\left(  \beta_{p}^{\mathbf{e}}\psi^{-1}\left(
\ell,s\right)  \right)  }{\partial s_{1}} & \cdots & \dfrac{\partial\left(
\beta_{p}^{\mathbf{e}}\psi^{-1}\left(  \ell,s\right)  \right)  }{\partial
s_{m}}%
\end{array}
\right]  ;
\]
$\left(  \ell,s,t\right)  \mapsto C\left(  \ell,s,t\right)  $ is a $m\times p$
matrix-valued function which is equal to
\[
\left[
\begin{array}
[c]{ccccc}%
\dfrac{\partial\left(  \beta_{p+1}^{\mathbf{e}}\psi^{-1}\left(  \ell
,s,t\right)  \right)  }{\partial\ell_{1}} & \cdots & \dfrac{\partial\left(
\beta_{p+1}^{\mathbf{e}}\psi^{-1}\left(  \ell,s,t\right)  \right)  }{\partial
s_{1}} & \cdots & \dfrac{\partial\left(  \beta_{p+1}^{\mathbf{e}}\psi
^{-1}\left(  \ell,s,t\right)  \right)  }{\partial s_{m}}\\
\vdots & \ddots & \vdots & \ddots & \vdots\\
\dfrac{\partial\left(  \beta_{n}^{\mathbf{e}}\psi^{-1}\left(  \ell,s,t\right)
\right)  }{\partial\ell_{1}} & \cdots & \dfrac{\partial\left(  \beta
_{n}^{\mathbf{e}}\psi^{-1}\left(  \ell,s,t\right)  \right)  }{\partial s_{1}}
& \cdots & \dfrac{\partial\left(  \beta_{n}^{\mathbf{e}}\psi^{-1}\left(
\ell,s,t\right)  \right)  }{\partial s_{m}}%
\end{array}
\right]
\]
and finally
\[
D\left(  \ell,s,t\right)  =\left[
\begin{array}
[c]{ccc}%
\dfrac{\partial\left(  \beta_{p+1}^{\mathbf{e}}\psi^{-1}\left(  \ell
,s,t\right)  \right)  }{\partial t_{1}} & \cdots & \dfrac{\partial\left(
\beta_{p+1}^{\mathbf{e}}\psi^{-1}\left(  \ell,s,t\right)  \right)  }{\partial
t_{m}}\\
\vdots & \ddots & \vdots\\
\dfrac{\partial\left(  \beta_{n}^{\mathbf{e}}\psi^{-1}\left(  \ell,s,t\right)
\right)  }{\partial t_{1}} & \cdots & \dfrac{\partial\left(  \beta
_{n}^{\mathbf{e}}\psi^{-1}\left(  \ell,s,t\right)  \right)  }{\partial t_{m}}%
\end{array}
\right]
\]
is of order $m.$ Now, since
\[
\det\left[
\begin{array}
[c]{c}%
A\left(  \ell,s\right)  \\
C\left(  \ell,s,t\right)
\end{array}%
\begin{array}
[c]{c}%
\\
D\left(  \ell,s,t\right)
\end{array}
\right]  =\det A\left(  \ell,s\right)  \times\det D\left(  \ell,s,t\right)
\neq0
\]
and because the Jacobian of $\beta$ is the submatrix $A\left(  \ell,s\right)
,$ it follows that its determinant does not vanish. Thus, the Jacobian of the
map $\beta$ has constant full rank as well. In order to establish that is a
diffeomorphism, we appeal to the fact that
\[
\iota\left(  \beta\left(  \lambda,s_{1},\cdots,s_{m}\right)  \right)
=\mathfrak{P}\left(  A(s)\right)  f
\]
and
\[
\left(  f,s_{1},\cdots,s_{m},t_{i_{1}},\cdots,t_{i_{m}}\right)  \mapsto\left[
\begin{array}
[c]{c}%
\mathfrak{P}\left(  A(s)\right)  f\\
\sigma\left(  t,\mathfrak{P}\left(  A(s)\right)  f\right)
\end{array}
\right]
\]
is a bijection. Thus it is clear that $\beta$ is smooth bijective with
constant full-rank. Thus, it is a diffeomorphism. This completes the proof.
\end{proof}

\begin{remark}
For $\lambda\in\Sigma$ there exist real numbers $f_{k}$ such that
$\lambda=\sum_{k\notin\mathbf{e}}f_{k}Z_{k}^{\ast}=f_{k_{1}}Z_{k_{1}}^{\ast
}+\cdots+f_{k_{n-2m}}Z_{k_{n-2m}}^{\ast}.$ Thus, defining
\[
\mathfrak{j}:f_{k_{1}}Z_{k_{1}}^{\ast}+\cdots+f_{k_{n-2m}}Z_{k_{n-2m}}^{\ast
}\mapsto\left(  f_{k_{1}},\cdots,f_{k_{n-2m}}\right)  ,
\]
it is clear that the unitary dual of the Lie group $N$ is parametrized by
$\mathfrak{j}\left(  \Sigma\right)  $ which is a Zariski open subset of
$\mathbb{R}^{\dim\Sigma}.$ Although this is an abuse of notation, we shall
make no distinction between $\mathfrak{j}\left(  \Sigma\right)  $ and
$\Sigma.$
\end{remark}

An example is now in order.

\begin{example}
\label{Ex}Let $\mathfrak{n}$ be a nilpotent Lie algebra spanned by
\[
Z_{1}=X_{1},Z_{2}=X_{2},Z_{4}=X_{3},A_{1}=X_{4}%
\]
with non-trivial Lie brackets%
\[
\left[  X_{4},X_{2}\right]  =2X_{1}\text{ and }\left[  X_{4},X_{3}\right]
=2X_{2}.
\]
Letting
\[
\lambda=\lambda_{1}X_{1}^{\ast}+\lambda_{2}X_{2}^{\ast}+\lambda_{3}X_{3}%
^{\ast}+\alpha X_{4}^{\ast}\in\mathfrak{n}^{\ast},
\]
the skew symmetric matrix-valued function $\mathbf{M}$ is computed as
follows:
\[
\mathbf{M}\left(  \lambda\right)  =\left[
\begin{array}
[c]{cccc}%
0 & 0 & 0 & 0\\
0 & 0 & 0 & -2\lambda\left(  X_{1}\right)  \\
0 & 0 & 0 & -2\lambda\left(  X_{2}\right)  \\
0 & 2\lambda\left(  X_{1}\right)   & 2\lambda\left(  X_{2}\right)   & 0
\end{array}
\right]  =\left[
\begin{array}
[c]{cccc}%
0 & 0 & 0 & 0\\
0 & 0 & 0 & -2\lambda_{1}\\
0 & 0 & 0 & -2\lambda_{2}\\
0 & 2\lambda_{1} & 2\lambda_{2} & 0
\end{array}
\right]  .
\]
We check that
\[
\lambda\mapsto\mathbf{e}\left(  \lambda\right)  =\left\{
\begin{array}
[c]{c}%
\left\{  2,4\right\}  \text{ if }\lambda\left(  X_{1}\right)  \neq0\\
\left\{  3,4\right\}  \text{ if }\lambda\left(  X_{1}\right)  =0\text{ and
}\lambda\left(  X_{2}\right)  \neq0\\
\left\{  {}\right\}  \text{ if }\lambda\left(  X_{1}\right)  =0\text{ and
}\lambda\left(  X_{2}\right)  =0
\end{array}
\right.  .
\]
Next,
\[
\Omega=\Omega_{\left\{  2,4\right\}  }=\left\{  \lambda=\lambda_{1}X_{1}%
^{\ast}+\lambda_{2}X_{2}^{\ast}+\lambda_{3}X_{3}^{\ast}+\alpha X_{4}^{\ast}%
\in\mathfrak{n}^{\ast}:\lambda_{1}\neq0\right\}
\]
is Zariski open, $N$-invariant and
\[
\Sigma=\left\{  \lambda\in\Omega_{\left\{  2,4\right\}  }:\lambda\left(
X_{2}\right)  =\lambda\left(  X_{4}\right)  =0\right\}  =\left(
\mathbb{R-}\left\{  0\right\}  \right)  \times\mathbb{R}%
\]
is a cross-section for the coadjoint orbits in $\Omega_{\left\{  2,4\right\}
}.$ We note that for every linear functional $\lambda\in\Sigma,$ the ideal
\[
\mathfrak{p}=\mathbb{R}X_{1}+\mathbb{R}X_{2}+\mathbb{R}X_{3}%
\]
is a polarizing algebra subordinated to $\lambda.$ Next,
\[
\mathfrak{P}\left(  A(t)\right)  =\left[
\begin{array}
[c]{ccc}%
1 & 0 & 0\\
-t & 1 & 0\\
\frac{1}{2}t^{2} & -t & 1
\end{array}
\right]
\]
and
\[
\beta\left(  \lambda,t\right)  =\beta\left(  \lambda_{1},\lambda_{3},t\right)
=\left(  \lambda_{1}\right)  X_{1}^{\ast}-\left(  2t\lambda_{1}\right)
X_{2}^{\ast}+\left(  2\lambda_{1}t^{2}+\lambda_{3}\right)  X_{3}^{\ast
}=\left(  \lambda_{1},-2t\lambda_{1},2\lambda_{1}t^{2}+\lambda_{3}\right)
\]
defines a diffeomorphism between the sets $\Sigma\times\mathbb{R}$ and
$\Omega_{\left\{  2,4\right\}  }.$ Additionally, it is clear that
$\beta\left(  \Sigma\times\mathbb{R}\right)  $ is a Zariski open subset of
$\mathbb{R}^{3}\ $and
\[
\beta\left(  \Sigma\times\mathbb{R}\right)  =\left(  \mathbb{R-}\left\{
0\right\}  \right)  \times\mathbb{R\times R}\text{.}%
\]

\end{example}

\begin{remark}
For each $\lambda\in\Sigma,$ the corresponding unitary irreducible
representation $\sigma_{\lambda}$ of $N$ (see (\ref{irreducible})) is obtained
by inducing the character $\chi_{\lambda}$ of the normal subgroup $P$ which is
defined as follows:
\[
\chi_{\lambda}\left(  \exp\left(  t_{1}Z_{1}+\cdots+t_{p}Z_{p}\right)
\right)  =e^{2\pi i\left\langle \lambda,t_{1}Z_{1}+\cdots+t_{p}Z_{p}%
\right\rangle }=e^{2\pi i\lambda\left(  t_{1}Z_{1}+\cdots+t_{p}Z_{p}\right)
}.
\]
Since $\exp\left(  \mathbb{R}A_{1}+\cdots+\mathbb{R}A_{m}\right)  $ is a
cross-section for $N/P$ in $N,$ we shall realize the unitary representation
$\sigma_{\lambda}$ as acting on the Hilbert space $L^{2}(N/P)=L^{2}\left(
\mathbb{R}^{m}\right)  $. Following the discussion in Subsection
\ref{realization}, it is easy to see that if $x=\left(  x_{1},\cdots
,x_{m}\right)  ,$
\[
A\left(  a\right)  =a_{1}A_{1}+\cdots+a_{m}A_{m},\text{ }Z\left(  t\right)
=t_{1}Z_{1}+\cdots+t_{p}Z_{p}%
\]
and $h\in L^{2}\left(  \mathbb{R}^{m}\right)  ,$ then for every linear
functional $\lambda\in\Sigma$
\begin{equation}
\left[  \sigma_{\lambda}\left(  \exp\left(  Z\left(  t\right)  \right)
\exp\left(  A\left(  a\right)  \right)  \right)  h\right]  \left(  x\right)
=e^{2\pi i\left\langle \lambda,e^{-adA\left(  x\right)  }Z\left(  t\right)
\right\rangle }h\left(  x_{1}-a_{1},\cdots,x_{m}-a_{m}\right)  .
\label{irreducibles}%
\end{equation}

\end{remark}

We remark that although the Plancherel measure for an arbitrary nilpotent Lie
group has already been computed in general form in the book of Corwin and
Greenleaf \cite{Corwin}, in order to prove the main results stated in the
introduction, we will need to establish a connection between the Plancherel
measure of $N$ and the determinant of the Jacobian of the map $\beta.$ To make
this connection as transparent and as clear as possible, we shall need the
following lemma.

\begin{lemma}
\label{Plancherel measure}Let $J_{\beta}\left(  \lambda,s_{1},\cdots
,s_{m}\right)  $ be the Jacobian of the smooth map $\beta$ defined in
(\ref{Malcev}). The Plancherel measure of $N$ is up to multiplication by a
constant equal to%
\begin{equation}
d\mu\left(  \lambda\right)  =\left\vert \det J_{\beta}\left(  \lambda
,0\right)  \right\vert d\lambda\label{Planch}%
\end{equation}
where $d\lambda$ is the Lebesgue measure on $\mathbb{R}^{\dim\Sigma}.$
\end{lemma}

\begin{proof}
Since the set of smooth functions of compact support is dense in $L^{2}\left(
N\right)  ,$ it suffices to show that for any smooth function $\mathbf{F}$ of
compact support on the group $N$,
\[
\int_{\Sigma}\left\Vert \widehat{\mathbf{F}}\left(  \sigma_{\lambda}\right)
\right\Vert _{\mathcal{HS}\left(  L^{2}\left(  \mathbb{R}^{m}\right)  \right)
}^{2}\left\vert \det J_{\beta}\left(  \lambda,0\right)  \right\vert
d\lambda=\left\Vert \mathbf{F}\right\Vert _{L^{2}\left(  N\right)  }^{2}.
\]
In order to simplify our presentation, we shall identify the set $N=PM$ with
$\mathbb{R}^{p}\times\mathbb{R}^{m}$ via the map
\[
\exp\left(  t_{1}Z_{1}+\cdots+t_{p}Z_{p}\right)  \exp\left(  a_{1}A_{1}%
+\cdots+a_{m}A_{m}\right)  \mapsto\left(  t_{1},\cdots,t_{p},a_{1}%
,\cdots,a_{m}\right)  =\left(  t,a\right)  .
\]
For any smooth function $\mathbf{F}$ of compact support on the group $N,$ the
operator $\widehat{\mathbf{F}}\left(  \sigma_{\lambda}\right)  $ (see
(\ref{Fourier})) is defined on $L^{2}\left(  \mathbb{R}^{m}\right)  $ as
follows. For $\phi\in L^{2}\left(  \mathbb{R}^{m}\right)  $ we have%
\begin{align}
\left[  \widehat{\mathbf{F}}\left(  \sigma_{\lambda}\right)  \phi\right]
\left(  x\right)   &  =\int_{\mathbb{R}^{p}}\int_{\mathbb{R}^{m}}%
\mathbf{F}\left(  t,a\right)  \left[  \left(  \sigma_{\lambda}\left(
t,a\right)  \right)  \phi\right]  \left(  x\right)  dadt\label{comeback}\\
&  =\int_{\mathbb{R}^{p}}\int_{\mathbb{R}^{m}}\mathbf{F}\left(  t,a\right)
e^{2\pi i\left\langle \lambda,e^{-adA\left(  x\right)  }Z\left(  t\right)
\right\rangle }\phi\left(  x-a\right)  dadt\\
&  =\int_{\mathbb{R}^{p}}\int_{\mathbb{R}^{m}}\mathbf{F}\left(  t,a\right)
e^{2\pi i\left\langle \exp\left(  A\left(  x\right)  \right)  \cdot
\lambda,Z\left(  t\right)  \right\rangle }\phi\left(  x-a\right)  dadt.
\end{align}
Next, we recall that $\mathfrak{P}\left(  A(x)\right)  f=\left[
e^{-adA\left(  x\right)  |\mathfrak{p}}\right]  ^{T}f$ and
\[
\iota\left(  \exp\left(  A\left(  x\right)  \right)  \cdot\lambda\right)
=\left[
\begin{array}
[c]{c}%
\mathfrak{P}\left(  A(x)\right)  f\\
0
\end{array}
\right]  \text{ where }\iota\left(  \lambda\right)  =\left[
\begin{array}
[c]{c}%
f\\
0
\end{array}
\right]  .
\]
Next,
\begin{align*}
\left[  \widehat{\mathbf{F}}\left(  \sigma_{\lambda}\right)  \phi\right]
\left(  x\right)   &  =\int_{\mathbb{R}^{p}}\int_{\mathbb{R}^{m}}%
\mathbf{F}\left(  t,a\right)  e^{2\pi i\left\langle \exp\left(  A\left(
x\right)  \right)  \cdot\lambda,Z\left(  t\right)  \right\rangle }\phi\left(
x-a\right)  dadt\\
&  =\int_{\mathbb{R}^{p}}\int_{\mathbb{R}^{m}}\mathbf{F}\left(  t,a\right)
e^{2\pi i\left\langle \beta\left(  \lambda,x\right)  ,Z\left(  t\right)
\right\rangle }\phi\left(  x-a\right)  dadt\\
&  =\int_{\mathbb{R}^{p}}\int_{\mathbb{R}^{m}}\mathbf{F}\left(  t,x-a\right)
e^{2\pi i\left\langle \beta\left(  \lambda,x\right)  ,Z\left(  t\right)
\right\rangle }\phi\left(  a\right)  \text{ }da\text{ }dt\\
&  =\int_{\mathbb{R}^{p}}\left(  \int_{\mathbb{R}^{m}}\mathbf{F}\left(
t,x-a\right)  e^{2\pi i\left\langle \beta\left(  \lambda,x\right)  ,Z\left(
t\right)  \right\rangle }\text{ }dt\right)  \text{ }\phi\left(  a\right)  da.
\end{align*}
Thus, $\widehat{\mathbf{F}}\left(  \sigma_{\lambda}\right)  $ is an integral
operator on $L^{2}\left(  \mathbb{R}^{m}\right)  $ with kernel $K_{\lambda
,\mathbf{F}}$ given by
\begin{equation}
K_{\lambda,\mathbf{F}}\left(  x,a\right)  =\int_{\mathbb{R}^{p}}%
\mathbf{F}\left(  t,x-a\right)  e^{2\pi i\left\langle \beta\left(
\lambda,x\right)  ,Z\left(  t\right)  \right\rangle }dt. \label{Kernel}%
\end{equation}
Now, let $\mathfrak{F}_{1}$ be the partial Euclidean Fourier transform in the
direction of $t.$ It is clear that
\[
K_{\lambda,\mathbf{F}}\left(  x,a\right)  =\left[  \mathfrak{F}_{1}%
\mathbf{F}\right]  \left(  \beta\left(  \lambda,x\right)  ,x-a\right)  .
\]
Additionally, the square of the Hilbert-Schmidt norm of the operator
$\widehat{\mathbf{F}}\left(  \sigma_{\lambda}\right)  $ is given by
\begin{align*}
\left\Vert \widehat{\mathbf{F}}\left(  \sigma_{\lambda}\right)  \right\Vert
_{\mathcal{HS}\left(  L^{2}\left(  \mathbb{R}^{m}\right)  \right)  }^{2}  &
=\int_{\mathbb{R}^{m}}\int_{\mathbb{R}^{m}}\left\vert K_{\lambda,\mathbf{F}%
}\left(  x,a\right)  \right\vert ^{2}dx\text{ }da\\
&  =\int_{\mathbb{R}^{m}}\int_{\mathbb{R}^{m}}\left\vert \mathfrak{F}%
_{1}\mathbf{F}\left(  \beta\left(  \lambda,x\right)  ,x-a\right)  \right\vert
^{2}dx\text{ }da\\
&  =\int_{\mathbb{R}^{m}}\int_{\mathbb{R}^{m}}\left\vert \mathfrak{F}%
_{1}\mathbf{F}\left(  \beta\left(  \lambda,x\right)  ,a\right)  \right\vert
^{2}dx\text{ }da.
\end{align*}
Observing that
\[
\iota\left(  \beta\left(  \lambda,x+t\right)  \right)  =\mathfrak{P}\left(
A(x)\right)  \mathfrak{P}\left(  A(t)\right)  f,
\]
the components of $\beta\left(  \lambda,x+t\right)  $ may be computed by
multiplying a unipotent matrix by the matrix representation of $\beta\left(
\lambda,t\right)  $, the determinant of the Jacobian of the map $\beta$ at
$\left(  \lambda,x\right)  $ is then given by
\begin{align*}
\det J_{\beta}\left(  \left(  \lambda,x+t\right)  \right)   &  =\det\left(
\left[  e^{\left(  -adx_{1}A_{1}-\cdots-adx_{m}A_{m}\right)  |_{\mathfrak{p}}%
}\right]  \right)  \det J_{\beta}\left(  \lambda,t\right) \\
&  =1\times\det J_{\beta}\left(  \lambda,t\right)  .
\end{align*}
It follows that
\[
\det J_{\beta}\left(  \lambda,x\right)  =\det J_{\beta}\left(  \lambda
,0\right)  =\mathbf{P}\left(  \lambda\right)
\]
where $\mathbf{P}\left(  \lambda\right)  $ is a polynomial in the coordinates
of $\lambda$. Next,
\begin{align*}
&  \int_{\Sigma}\left\Vert \widehat{\mathbf{F}}\left(  \sigma_{\lambda
}\right)  \right\Vert _{\mathcal{H}S\left(  L^{2}\left(  \mathbb{R}%
^{m}\right)  \right)  }^{2}\left\vert \det J_{\beta}\left(  \lambda,0\right)
\right\vert d\lambda\\
&  =\int_{\Sigma}\left(  \int_{%
\mathbb{R}
^{m}}\int_{%
\mathbb{R}
^{m}}\left\vert \left(  \mathfrak{F}_{1}\mathbf{F}\right)  \left(
\beta\left(  \lambda,x\right)  ,a\right)  \right\vert ^{2}da\text{ }dx\text{
}\left\vert \det J_{\beta}\left(  \lambda,0\right)  \right\vert \right)
d\lambda\\
&  =\int_{\Sigma}\int_{%
\mathbb{R}
^{m}}\int_{%
\mathbb{R}
^{m}}\left\vert \left(  \mathfrak{F}_{1}\mathbf{F}\right)  \left(
\beta\left(  \lambda,x\right)  ,a\right)  \right\vert ^{2}da\text{ }dx\text{
}\left\vert \det J_{\beta}\left(  \lambda,0\right)  \right\vert \text{
}d\lambda\\
&  =\int_{\Sigma\times%
\mathbb{R}
^{m}}\int_{%
\mathbb{R}
^{m}}\left\vert \left(  \mathfrak{F}_{1}\mathbf{F}\right)  \left(
\beta\left(  \lambda,x\right)  ,a\right)  \right\vert ^{2}da\text{ }\left\vert
\det J_{\beta}\left(  \lambda,0\right)  \right\vert \text{ }d\left(
\lambda,x\right) \\
&  =\int_{\Sigma\times%
\mathbb{R}
^{m}}\int_{%
\mathbb{R}
^{m}}\left\vert \left(  \mathfrak{F}_{1}\mathbf{F}\right)  \left(
\beta\left(  \lambda,x\right)  ,a\right)  \right\vert ^{2}da\text{ }d\left(
\beta\left(  \lambda,x\right)  \right) \\
&  =\int_{\Omega}\int_{%
\mathbb{R}
^{m}}\left\vert \left(  \mathfrak{F}_{1}\mathbf{F}\right)  \left(  z,a\right)
\right\vert ^{2}da\text{ }dz.
\end{align*}
Next, appealing to Plancherel's theorem
\[
\int_{\Sigma}\left\Vert \widehat{\mathbf{F}}\left(  \sigma_{\lambda}\right)
\right\Vert _{\mathcal{H}S\left(  L^{2}\left(  \mathbb{R}^{m}\right)  \right)
}^{2}\left\vert \det J_{\beta}\left(  \lambda,0\right)  \right\vert
d\lambda=\int_{%
\mathbb{R}
^{p}}\int_{%
\mathbb{R}
^{m}}\left\vert \mathbf{F}\left(  z,a\right)  \right\vert ^{2}dadz=\left\Vert
\mathbf{F}\right\Vert _{L^{2}\left(  N\right)  }^{2}.
\]

\end{proof}

We shall now define a transform which plays an important role in proving our
main results. Let $\mathbf{A}$ be a $d\mu$-measurable subset of $\mathbb{R}%
^{n-2m}.$ Define (in a formal way) the map
\[
J_{\mathbf{A}}:L^{2}\left(  \mathbf{A}\times\mathbb{R}^{m},d\mu\left(
\lambda\right)  \right)  \rightarrow l^{2}\left(  \Gamma\right)
\]
such that for $l=\left(  l_{1},\cdots,l_{m}\right)  \in\mathbb{Z}^{m},$%
\[
Z\left(  k\right)  =\sum_{j=1}^{p}k_{j}Z_{j}\in\sum_{j=1}^{p}\mathbb{Z}%
Z_{j}\text{ and }A\left(  l\right)  =\sum_{j=1}^{m}l_{j}A_{j}\in\sum_{j=1}%
^{m}\mathbb{Z}A_{j}%
\]
we have
\begin{equation}
\left[  J_{\mathbf{A}}F\right]  \left(  \exp\left(  Z\left(  k\right)
\right)  \exp\left(  A\left(  l\right)  \right)  \right)  =\int_{\mathbf{A}%
}\int_{\left[  0,1\right)  ^{m}}F\left(  \lambda,t-l\right)  e^{2\pi
i\left\langle \beta\left(  \lambda,t\right)  ,Z\left(  k\right)  \right\rangle
}dt\text{ }d\mu\left(  \lambda\right)  \label{map J}%
\end{equation}
where $k\in\mathbb{Z}^{m}.$

Let $\mathfrak{F}_{\mathbb{R}^{p}/\mathbb{Z}^{p}}$ be the Fourier transform
defined on $L^{2}\left(  \mathbb{R}^{p}/\mathbb{Z}^{p}\right)  .$ Let $1_{X}$
denotes the indicator function for a given set $X.$

\begin{proposition}
\label{convergence}Assume that $\mathbf{A}$ is a $d\mu$-measurable bounded
subset of $\mathbb{R}^{n-2m}$. If $H\in L^{2}\left(  \mathbf{A}\times
\mathbb{R}^{m}\right)  $ is a smooth function of compact support, then
$J_{\mathbf{A}}H\in l^{2}\left(  \Gamma\right)  .$
\end{proposition}

\begin{proof}
First, observe that
\[
\left[  J_{\mathbf{A}}H\right]  \left(  \exp\left(  Z\left(  k\right)
\right)  \exp\left(  A\left(  l\right)  \right)  \right)  =\int_{\mathbf{A}%
}\int_{\left[  0,1\right)  ^{m}}H\left(  \lambda,t-l\right)  \text{ }e^{2\pi
i\left\langle \beta\left(  \lambda,t\right)  ,Z\left(  k\right)  \right\rangle
}dt\text{ }d\mu\left(  \lambda\right)
\]
and for a fixed $\sum_{j=1}^{p}k_{j}Z_{j},$ the sequence
\[
\left(  \left[  J_{\mathbf{A}}H\right]  \left(  \exp\left(  Z\left(  k\right)
\right)  \exp\left(  A\left(  l\right)  \right)  \right)  \right)
_{l\in\mathbb{Z}^{m}}%
\]
has compact support. Making the change of variable $s=t-l$ we obtain that%
\[
\left[  J_{\mathbf{A}}H\right]  \left(  \exp\left(  Z\left(  k\right)
\right)  \exp\left(  A\left(  l\right)  \right)  \right)  =\int_{\mathbf{A}%
}\int_{\left[  0,1\right)  ^{m}-l}H\left(  \lambda,s\right)  \text{ }e^{2\pi
i\left\langle \beta\left(  \lambda,s+l\right)  ,Z\left(  k\right)
\right\rangle }ds\text{ }d\mu\left(  \lambda\right)  .
\]
Next, since
\[
\beta\left(  \lambda,s+l\right)  =\exp\left(  A\left(  s\right)  \right)
\exp\left(  A\left(  l\right)  \right)  \cdot\lambda
\]
and
\[
d\mu\left(  \lambda\right)  =\left\vert \det J_{\beta\left(  \lambda,0\right)
}\right\vert d\lambda
\]
it follows that
\begin{align*}
\left[  J_{\mathbf{A}}H\right]  \left(  \exp\left(  Z\left(  k\right)
\right)  \exp\left(  A\left(  l\right)  \right)  \right)   &  =\int
_{\mathbf{A\times}\left(  \left[  0,1\right)  ^{m}-l\right)  }H\left(
\lambda,s\right)  \text{ }e^{2\pi i\left\langle \beta\left(  \lambda
,s+l\right)  ,Z\left(  k\right)  \right\rangle }\text{ }\left\vert \det
J_{\beta\left(  \lambda,0\right)  }\right\vert \text{ }d\left(  \lambda
,s\right)  \\
&  =\int_{\mathbf{A\times}\left(  \left[  0,1\right)  ^{m}-l\right)  }H\left(
\lambda,s\right)  \text{ }e^{2\pi i\left\langle \beta\left(  \lambda,s\right)
,e^{-adA\left(  l\right)  }Z\left(  k\right)  \right\rangle }\text{
}\left\vert \det J_{\beta\left(  \lambda,0\right)  }\right\vert \text{
}d\left(  \lambda,s\right)  .
\end{align*}
The last equality above is due to the equation%
\[
\left\langle \exp\left(  \sum_{j=1}^{m}s_{j}A_{j}\right)  \cdot\omega,Z\left(
k\right)  \right\rangle =\left\langle \omega,e^{-ad\left(  \sum_{j=1}^{m}%
s_{j}A_{j}\right)  }Z\left(  k\right)  \right\rangle \text{ for }\omega
\in\Sigma.
\]
Next, the change of variable $r=\beta\left(  \lambda,s\right)  $ yields
\[
\left[  J_{\mathbf{A}}H\right]  \left(  \exp\left(  Z\left(  k\right)
\right)  \exp\left(  A\left(  l\right)  \right)  \right)  =\int_{\beta\left(
\mathbf{A\times}\left(  \left[  0,1\right)  ^{m}-l\right)  \right)  }\left(
H\circ\beta^{-1}\right)  \left(  r\right)  \text{ }e^{2\pi i\left\langle
r,e^{-adA\left(  l\right)  }Z\left(  k\right)  \right\rangle }\text{ }dr
\]
where $dr$ is the Lebesgue measure on $\mathbb{R}^{p}.$ Next for each $l,$ we
write $\beta\left(  \mathbf{A\times}\left(  \left[  0,1\right)  ^{m}-l\right)
\right)  $ as a finite disjoint union of subsets of $\mathbb{R}^{p}$; each
contained in a fundamental domain of $\mathbb{Z}^{p}$ as follows:
\[
\beta\left(  \mathbf{A\times}\left(  \left[  0,1\right)  ^{m}-l\right)
\right)  =\text{ }\overset{\cdot}{%
{\displaystyle\bigcup\limits_{j\in J\left(  H,l\right)  }}
}\left(  K_{\mathbf{A},l}+j\right)  \text{ \ \ \ \ \ \ \ \ \ \ \ \ (where
}J(H,l)\subset\mathbb{Z}^{p}\text{)}.
\]
Letting $1_{K_{\mathbf{A},l}+j}$ be the indicator function of the set
$K_{\mathbf{A},l}+j,$ we obtain
\begin{align*}
\left[  J_{\mathbf{A}}H\right]  \left(  \exp\left(  Z\left(  k\right)
\right)  \exp\left(  A\left(  l\right)  \right)  \right)   &  =\sum_{j\in
J\left(  H,l\right)  }\int_{K_{\mathbf{A},l}+j}\left(  H\circ\beta
^{-1}\right)  \left(  r\right)  \text{ }e^{2\pi i\left\langle r,e^{-adA\left(
l\right)  }Z\left(  k\right)  \right\rangle }\text{ }dr\\
&  =\sum_{j\in J\left(  H,l\right)  }\left[  \mathfrak{F}_{\mathbb{R}%
^{p}/\mathbb{Z}^{p}}\left(  \left(  H\circ\beta^{-1}\right)  \times
1_{K_{\mathbf{A},l}+j}\right)  \right]  \left(  e^{-adA\left(  l\right)
}Z\left(  k\right)  \right)
\end{align*}
where for each $l\in\mathbb{Z}^{m}$
\[
\sum_{j\in J\left(  H,l\right)  }\left[  \mathfrak{F}_{\mathbb{R}%
^{p}/\mathbb{Z}^{p}}\left(  \left(  H\circ\beta^{-1}\right)  \times
1_{K_{\mathbf{A},l}+j}\right)  \right]  \left(  e^{-adA\left(  l\right)
}Z\left(  k\right)  \right)
\]
is a finite sum of Fourier transforms of smooth functions of compact support.
In order to avoid cluster of notation, we set%
\begin{equation}
e^{-adA\left(  l\right)  }Z\left(  k\right)  =e^{-adA\left(  l\right)
}\left(  \sum_{j=1}^{p}k_{j}Z_{j}\right)  =a_{l}\left(  k\right)  .
\end{equation}
Finally,
\begin{align*}
\left\Vert \left[  J_{\mathbf{A}}H\right]  \right\Vert _{l^{2}\left(
\Gamma\right)  }^{2} &  =\sum_{\left(  k,l\right)  \in\mathbb{Z}^{p}%
\times\mathbb{Z}^{m}}\left\vert J_{\mathbf{A}}H\left(  \exp\left(  Z\left(
k\right)  \right)  \exp\left(  A\left(  l\right)  \right)  \right)
\right\vert ^{2}\\
&  =\sum_{k\in\mathbb{Z}^{p}}\sum_{l\in\mathbb{Z}^{m}}\left\vert
J_{\mathbf{A}}H\left(  \exp\left(  Z\left(  k\right)  \right)  \exp\left(
A\left(  l\right)  \right)  \right)  \right\vert ^{2}\\
&  =\sum_{k\in\mathbb{Z}^{p}}\sum_{l\in F}\left\vert \sum_{j\in J\left(
H,l\right)  }\left[  \mathfrak{F}_{\mathbb{R}^{p}/\mathbb{Z}^{p}}\left(
\left(  H\circ\beta^{-1}\right)  \times1_{K_{\mathbf{A},l}+j}\right)  \right]
\left(  a_{l}\left(  k\right)  \right)  \right\vert ^{2}%
\end{align*}
where $F$ is a finite subset of $\mathbb{Z}^{m}$ and
\begin{align*}
\left\Vert \left[  J_{\mathbf{A}}H\right]  \right\Vert _{l^{2}\left(
\Gamma\right)  }^{2} &  =\sum_{l\in F}\sum_{k\in\mathbb{Z}^{p}}\left\vert
\sum_{j\in J\left(  H,l\right)  }\left[  \mathfrak{F}_{\mathbb{R}%
^{p}/\mathbb{Z}^{p}}\left(  \left(  H\circ\beta^{-1}\right)  \times
1_{K_{\mathbf{A},l}+j}\right)  \right]  \left(  a_{l}\left(  k\right)
\right)  \right\vert ^{2}.\\
&  \leq\sum_{l\in F}\sum_{k\in\mathbb{Z}^{p}}\left(  \sum_{j\in J\left(
H,l\right)  }\left\vert \left(  \mathfrak{F}_{\mathbb{R}^{p}/\mathbb{Z}^{p}%
}\left(  \overset{=\Theta_{H,l,j}}{\overbrace{\left(  H\circ\beta^{-1}\right)
\times1_{K_{\mathbf{A},l}+j}}}\right)  \right)  \left(  a_{l}\left(  k\right)
\right)  \right\vert \right)  ^{2}.
\end{align*}
Put
\[
S_{H}\left(  l\right)  =\sum_{k\in\mathbb{Z}^{p}}\left(  \sum_{j\in J\left(
H,l\right)  }\left\vert \left(  \mathfrak{F}_{\mathbb{R}^{p}/\mathbb{Z}^{p}%
}\Theta_{H,l,j}\right)  \left(  a_{l}\left(  k\right)  \right)  \right\vert
\right)  ^{2}.
\]
Expanding the inner sum
\[
\left(  \sum_{j\in J\left(  H,l\right)  }\left\vert \left(  \mathfrak{F}%
_{\mathbb{R}^{p}/\mathbb{Z}^{p}}\Theta_{H,l,j}\right)  \left(  a_{l}\left(
k\right)  \right)  \right\vert \right)  ^{2}%
\]
we obtain that
\begin{align*}
S_{H}\left(  l\right)   &  =\sum_{k\in\mathbb{Z}^{p}}\sum_{j\in J\left(
H,l\right)  }\left\vert \left(  \mathfrak{F}_{\mathbb{R}^{p}/\mathbb{Z}^{p}%
}\Theta_{H,l,j}\right)  \left(  a_{l}\left(  k\right)  \right)  \right\vert
^{2}\\
&  +2\sum_{j\neq j^{\prime}\text{and }j,j^{\prime}\in J\left(  H,l\right)
}\sum_{k\in\mathbb{Z}^{p}}\left\vert \left(  \mathfrak{F}_{\mathbb{R}%
^{p}/\mathbb{Z}^{p}}\Theta_{H,l,j}\right)  \left(  a_{l}\left(  k\right)
\right)  \left(  \mathfrak{F}_{\mathbb{R}^{p}/\mathbb{Z}^{p}}\Theta
_{H,l,j^{\prime}}\right)  \left(  a_{l}\left(  k\right)  \right)  \right\vert
\\
&  =\sum_{j\in J\left(  H,l\right)  }\left\Vert \Theta_{H,l,j}\right\Vert
^{2}\\
&  +2\sum_{j\neq j^{\prime}\text{and }j,j^{\prime}\in J\left(  H,l\right)
}\sum_{k\in\mathbb{Z}^{p}}\left\vert \left(  \mathfrak{F}_{\mathbb{R}%
^{p}/\mathbb{Z}^{p}}\Theta_{H,l,j}\right)  \left(  a_{l}\left(  k\right)
\right)  \times\left(  \mathfrak{F}_{\mathbb{R}^{p}/\mathbb{Z}^{p}}%
\Theta_{H,l,j^{\prime}}\right)  \left(  a_{l}\left(  k\right)  \right)
\right\vert \\
&  =\sum_{j\in J\left(  H,l\right)  }\left\Vert \Theta_{H,l,j}\right\Vert
_{L^{2}\left(  \mathbb{R}^{p}/\mathbb{Z}^{p}\right)  }^{2}+2\sum_{j\neq
j^{\prime}\text{and }j,j^{\prime}\in J\left(  H,l\right)  }\left\langle
\Theta_{H,l,j},\Theta_{H,l,j^{\prime}}\right\rangle _{L^{2}\left(
\mathbb{R}^{p}/\mathbb{Z}^{p}\right)  }.
\end{align*}
Since $J\left(  H,l\right)  $ is a finite set, appealing to the fact that each
$\Theta_{H,l,j}\ $is square-integrable over a fundamental domain of
$\mathbb{Z}^{p}$ we obtain the desired result
\[
\left\Vert \left[  J_{\mathbf{A}}H\right]  \right\Vert _{l^{2}\left(
\Gamma\right)  }^{2}\leq\text{ }\underset{\text{finite sum}}{\underbrace
{\sum_{l\in F}\sum_{j\in J\left(  H,l\right)  }}}\left\Vert \Theta
_{H,l,j}\right\Vert _{L^{2}\left(  \mathbb{R}^{p}/\mathbb{Z}^{p}\right)  }%
^{2}+\underset{\text{finite sum}}{\underbrace{\sum_{l\in F}\sum_{j\neq
j^{\prime}\text{and }j,j^{\prime}\in J\left(  H,l\right)  }}}2\left\langle
\Theta_{H,l,j},\Theta_{H,l,j^{\prime}}\right\rangle _{L^{2}\left(
\mathbb{R}^{p}/\mathbb{Z}^{p}\right)  }.
\]
Thus $\left\Vert \left[  J_{\mathbf{A}}H\right]  \right\Vert _{l^{2}\left(
\Gamma\right)  }^{2}$ is finite. 
\end{proof}

Let $\tau$ be the unitary representation of $\Gamma$ which acts on the Hilbert
space $L^{2}\left(  \mathbf{A}\times\mathbb{R}^{m},d\mu\left(  \lambda\right)
\right)  $ as follows:
\[
\left[  \tau\left(  \gamma\right)  F\right]  \left(  \lambda,t\right)
=\sigma_{\lambda}\left(  \gamma\right)  F\left(  \lambda,t\right)  .
\]
We shall prove the following three important facts.

\begin{itemize}
\item $J_{\mathbf{A}}$ intertwines $\tau$ with the right regular
representation of the discrete uniform group $\Gamma$ (Lemma \ref{intertwines})

\item If $\left\vert \beta\left(  \mathbf{A}\times\left[  0,1\right)
^{m}\right)  \right\vert >0$ \ and if $\beta\left(  \mathbf{A}\times\left[
0,1\right)  ^{m}\right)  $ is contained in a fundamental domain of
$\mathbb{Z}^{p}$ then $J_{\mathbf{A}}$ defines an isometry on a dense subset
of $L^{2}\left(  \mathbf{A}\times\mathbb{R}^{m},d\mu\left(  \lambda\right)
\right)  $ into $l^{2}\left(  \Gamma\right)  $ which extends uniquely to an
isometry of $L^{2}\left(  \mathbf{A}\times\mathbb{R}^{m},d\mu\left(
\lambda\right)  \right)  $ into $l^{2}\left(  \Gamma\right)  $ (Lemma
\ref{isometry})

\item If $\left\vert \beta\left(  \mathbf{A}\times\left[  0,1\right)
^{m}\right)  \right\vert >0$ \ and if $\beta\left(  \mathbf{A}\times\left[
0,1\right)  ^{m}\right)  $ is up to a null set equal to a fundamental domain
of $\mathbb{Z}^{p}$ then $J_{\mathbf{A}}$ defines a unitary map (Lemma
\ref{unit})
\end{itemize}

\begin{lemma}
\label{intertwines}The map $J_{\mathbf{A}}$ intertwines $\tau$ with the right
regular representation of $\Gamma.$
\end{lemma}

\begin{proof}
Let $F\in L^{2}\left(  \mathbf{A}\times\mathbb{R}^{m},d\mu\left(
\lambda\right)  \right)  $ such that $F$ is smooth with compact support. Put
\begin{align*}
A\left(  l\right)   &  =l_{1}A_{1}+\cdots+l_{m}A_{m}\text{ and }Z\left(
k\right)  =k_{1}Z_{1}+\cdots+k_{p}Z_{p}\\
A\left(  l^{\prime}\right)   &  =l_{1}^{\prime}A_{1}+\cdots+l_{m}^{\prime
}A_{m}\text{ and }Z\left(  k^{\prime}\right)  =k_{1}^{\prime}Z_{1}%
+\cdots+k_{p}^{\prime}Z_{p}.
\end{align*}
Firstly,%
\begin{align*}
&  \left[  J_{\mathbf{A}}\tau\left(  \exp A\left(  l^{\prime}\right)  \right)
F\right]  \left(  \exp Z\left(  k\right)  \exp A\left(  l\right)  \right)  \\
&  =\int_{\mathbf{A}}\int_{\left[  0,1\right)  ^{m}}\left[  \sigma_{\lambda
}\left(  \exp A\left(  l^{\prime}\right)  \right)  F\right]  \left(
\lambda,t-l\right)  e^{2\pi i\left\langle \beta\left(  \lambda,t\right)
,Z\left(  k\right)  \right\rangle }dtd\mu\left(  \lambda\right)  \\
&  =\int_{\mathbf{A}}\int_{\left[  0,1\right)  ^{m}}F\left(  \lambda,t-\left(
l+l^{\prime}\right)  \right)  e^{2\pi i\left\langle \beta\left(
\lambda,t\right)  ,Z\left(  k\right)  \right\rangle }dtd\mu\left(
\lambda\right)  \\
&  =J_{\mathbf{A}}F\left(  \left(  \exp Z\left(  k\right)  \exp A\left(
l+l^{\prime}\right)  \right)  \right)  \\
&  =R\left(  \exp A\left(  l^{\prime}\right)  \right)  J_{\mathbf{A}}F\left(
\exp Z\left(  k\right)  \exp A\left(  l\right)  \right)  .
\end{align*}
Secondly,%
\begin{align*}
&  \left[  J_{\mathbf{A}}\tau\left(  \exp Z\left(  k^{\prime}\right)  \right)
F\right]  \left(  \exp Z\left(  k\right)  \exp A\left(  l\right)  \right)  \\
&  =\int_{\mathbf{A}}\int_{\left[  0,1\right)  ^{m}}\left[  \sigma_{\lambda
}\left(  \exp Z\left(  k^{\prime}\right)  \right)  F\right]  \left(
\lambda,t-l\right)  e^{2\pi i\left\langle \beta\left(  \lambda,t\right)
,Z\left(  k\right)  \right\rangle }dtd\mu\left(  \lambda\right)  \\
&  =\int_{\mathbf{A}}\int_{\left[  0,1\right)  ^{m}}e^{2\pi i\left\langle
\lambda,e^{-ad\left(  t_{1}-l_{1}\right)  A_{1}+\cdots+\left(  t_{m}%
-l_{m}\right)  A_{m}}Z^{\prime}\right\rangle }F\left(  \lambda,t-l\right)
e^{2\pi i\left\langle \beta\left(  \lambda,t\right)  ,Z\left(  k\right)
\right\rangle }dtd\mu\left(  \lambda\right)  .
\end{align*}
Finally,
\begin{align*}
\left[  J_{\mathbf{A}}\tau\left(  \exp Z\left(  k^{\prime}\right)  \right)
F\right]  \left(  \exp Z\left(  k\right)  \exp A\left(  l\right)  \right)   &
=\int_{\mathbf{A}}\int_{\left[  0,1\right)  ^{m}}e^{2\pi i\left\langle
\beta\left(  \lambda,t\right)  ,Z\left(  k\right)  +e^{A}Z\left(  k^{\prime
}\right)  \right\rangle }F\left(  \lambda,t-l\right)  dtd\mu\left(
\lambda\right)  \\
&  =JF\left(  \exp\left(  Z\left(  k\right)  +e^{adA}Z\left(  k^{\prime
}\right)  \right)  \exp A\left(  l\right)  \right)  \\
&  =\left[  R\left(  \exp Z\left(  k^{\prime}\right)  \right)  JF\right]
\left(  \exp Z\left(  k\right)  \exp A\left(  l\right)  \right)  .
\end{align*}
This completes the proof.
\end{proof}

\begin{lemma}
\label{isometry}If $\beta\left(  \mathbf{A}\times\left[  0,1\right)
^{m}\right)  $ has positive Lebesgue measure in $%
\mathbb{R}
^{p}$ and is contained in a fundamental domain of $\mathbb{Z}^{p}$ then
$J_{\mathbf{A}}$ defines an isometry on a dense subset of $L^{2}\left(
\mathbf{A}\times\mathbb{R}^{m},d\mu\left(  \lambda\right)  \right)  $ into
$l^{2}\left(  \Gamma\right)  $ which extends uniquely to an isometry of
$L^{2}\left(  \mathbf{A}\times\mathbb{R}^{m},d\mu\left(  \lambda\right)
\right)  $ into $l^{2}\left(  \Gamma\right)  .$
\end{lemma}

\begin{proof}
Let $F\in L^{2}\left(  \mathbf{A}\times\mathbb{R}^{m},d\mu\left(
\lambda\right)  \right)  .$ Furthermore, let us assume that $F$ is smooth with
compact support in $\mathbf{A}\times\mathbb{R}^{m}$. Computing the norm of
$F,$ we obtain%
\begin{align}
\left\Vert F\right\Vert _{L^{2}\left(  \mathbf{A}\times\mathbb{R}^{m}%
,d\mu\left(  \lambda\right)  \right)  }^{2}  &  =\int_{\mathbf{A}}%
\int_{\mathbb{R}^{m}}\left\vert F\left(  \lambda,t\right)  \right\vert
^{2}dt\text{ }d\mu\left(  \lambda\right) \nonumber\\
&  =\int_{\mathbf{A}}\sum_{l\in\mathbb{Z}^{m}}\int_{\left[  0,1\right)  ^{m}%
}\left\vert F\left(  \lambda,t-l\right)  \right\vert ^{2}dt\text{ }d\mu\left(
\lambda\right)  . \label{norm}%
\end{align}
Letting $G_{l}\left(  \lambda,t\right)  =F\left(  \lambda,t-l\right)  ,$
making the change of variable $s=\beta\left(  \lambda,t\right)  $ and using
the fact that $\beta$ is a diffeomorphism (see Lemma \ref{beta 1}) we obtain%
\begin{align*}
\int_{\mathbf{A}}\sum_{l\in\mathbb{Z}^{m}}\int_{\left[  0,1\right)  ^{m}%
}\left\vert F\left(  \lambda,t-l\right)  \right\vert ^{2}dt\text{ }d\mu\left(
\lambda\right)   &  =\sum_{l\in\mathbb{Z}^{m}}\int_{\left[  0,1\right)  ^{m}%
}\int_{\mathbf{A}}\left\vert G_{l}\left(  \lambda,t\right)  \right\vert
^{2}d\mu\left(  \lambda\right)  dt\text{ }\\
&  =\sum_{l\in\mathbb{Z}^{m}}\int_{\mathbf{A\times}\left[  0,1\right)  ^{m}%
}\left\vert G_{l}\left(  \lambda,t\right)  \right\vert ^{2}\left\vert \det
J_{\beta}\left(  \lambda,0\right)  \right\vert d\left(  \lambda,t\right)
\text{ }\\
&  =\sum_{l\in\mathbb{Z}^{m}}\int_{\beta\left(  \mathbf{A\times}\left[
0,1\right)  ^{m}\right)  }\left\vert G_{l}\beta^{-1}\left(  s\right)
\right\vert ^{2}\left\vert \det J_{\beta}\left(  \lambda,0\right)  \right\vert
d\left(  \beta^{-1}\left(  s\right)  \right) \\
&  =\sum_{l\in\mathbb{Z}^{m}}\int_{\beta\left(  \mathbf{A\times}\left[
0,1\right)  ^{m}\right)  }\left\vert G_{l}\beta^{-1}\left(  s\right)
\right\vert ^{2}ds.
\end{align*}
Since $\beta\left(  \mathbf{A}\times\left[  0,1\right)  ^{m}\right)  $ is
contained in a fundamental domain of $\mathbb{Z}^{p}$ then
\begin{align*}
\int_{\mathbf{A}}\sum_{l\in\mathbb{Z}^{m}}\int_{\left[  0,1\right)  ^{m}%
}\left\vert F\left(  \lambda,t-l\right)  \right\vert ^{2}dt\text{ }d\mu\left(
\lambda\right)   &  =\sum_{l\in\mathbb{Z}^{m}}\left\Vert s\mapsto G_{l}%
\beta^{-1}\left(  s\right)  \right\Vert ^{2}\\
&  =\sum_{l\in\mathbb{Z}^{m}}\left\Vert \mathfrak{F}_{\mathbb{R}%
^{p}/\mathbb{Z}^{p}}\left(  G_{l}\beta^{-1}\right)  \right\Vert ^{2}\\
&  =\sum_{l\in\mathbb{Z}^{m}}\sum_{k\in\mathbb{Z}^{p}}\left\vert
\mathfrak{F}_{\mathbb{R}^{p}/\mathbb{Z}^{p}}\left(  G_{l}\beta^{-1}\right)
\left(  k\right)  \right\vert ^{2}\\
&  =\sum_{l\in\mathbb{Z}^{m}}\sum_{k\in\mathbb{Z}^{p}}\left\vert \int
_{\beta\left(  \mathbf{A\times}\left[  0,1\right)  ^{m}\right)  }\left(
G_{l}\beta^{-1}\right)  \left(  s\right)  \exp\left(  2\pi i\left\langle
s,k\right\rangle \right)  ds\right\vert ^{2}.
\end{align*}
Thus
\begin{equation}
\left\Vert F\right\Vert _{L^{2}\left(  \mathbf{A}\times\mathbb{R}^{m}%
,d\mu\left(  \lambda\right)  \right)  }^{2}=\sum_{l\in\mathbb{Z}^{m}}%
\sum_{k\in\mathbb{Z}^{p}}\left\vert \int_{\beta\left(  \mathbf{A\times}\left[
0,1\right)  ^{m}\right)  }G_{l}\beta^{-1}\left(  s\right)  \exp\left(  2\pi
i\left\langle s,k\right\rangle \right)  ds\right\vert ^{2}.
\end{equation}
Next, since $\beta\left(  \mathbf{A\times}\left[  0,1\right)  ^{m}\right)  $
is contained in a fundamental domain of $\mathbb{Z}^{p}$, the trigonometric
system
\[
\left\{  \chi_{\beta\left(  \mathbf{A\times}\left[  0,1\right)  ^{m}\right)
}\left(  s\right)  \times\exp\left(  2\pi i\left\langle s,k\right\rangle
\right)  :k\in%
\mathbb{Z}
^{p}\right\}
\]
forms a Parseval frame in $L^{2}\left(  \beta\left(  \mathbf{A\times}\left[
0,1\right)  ^{m}\right)  \right)  .$ Clearly this is true because the
orthogonal projection of an orthonormal basis is always a Parseval frame.
Letting
\[
\widehat{G_{l}\beta^{-1}}=\mathfrak{F}_{L^{2}\left(  \beta\left(
\mathbf{A\times}\left[  0,1\right)  ^{m}\right)  \right)  }\left(  s\mapsto
G_{l}\left(  \beta^{-1}\left(  s\right)  \right)  \right)
\]
be the Fourier transform of the function
\[
s\mapsto G_{l}\left(  \beta^{-1}\left(  s\right)  \right)  \in L^{2}\left(
\beta\left(  \mathbf{A\times}\left[  0,1\right)  ^{m}\right)  \right)
\]
it follows that
\[
\sum_{l\in\mathbb{Z}^{m}}\sum_{k\in\mathbb{Z}^{p}}\left\vert \int
_{\beta\left(  \mathbf{A\times}\left[  0,1\right)  ^{m}\right)  }G_{l}\left(
\beta^{-1}\left(  s\right)  \right)  \exp\left(  2\pi i\left\langle
s,k\right\rangle \right)  ds\right\vert ^{2}=\sum_{l\in\mathbb{Z}^{m}}%
\sum_{k\in\mathbb{Z}^{p}}\left\vert \widehat{G_{l}\beta^{-1}}\left(  k\right)
\right\vert ^{2}.
\]
Next,
\[
\sum_{l\in\mathbb{Z}^{m}}\sum_{k\in\mathbb{Z}^{p}}\left\vert \widehat
{G_{l}\beta^{-1}}\left(  k\right)  \right\vert ^{2}=\sum_{l\in\mathbb{Z}^{m}%
}\left\Vert \widehat{G_{l}\beta^{-1}}\right\Vert _{l^{2}\left(  \mathbb{Z}%
^{p}\right)  }^{2}=\sum_{l\in\mathbb{Z}^{m}}\int_{\beta\left(  \mathbf{A\times
}\left[  0,1\right)  ^{m}\right)  }\left\vert G_{l}\left(  \beta^{-1}\left(
s\right)  \right)  \right\vert ^{2}ds.
\]
Now substituting $\left(  \lambda,t\right)  $ for $\beta^{-1}\left(  s\right)
$,%
\begin{align}
&  \sum_{l\in\mathbb{Z}^{m}}\sum_{k\in\mathbb{Z}^{p}}\left\vert \int
_{\beta\left(  \mathbf{A\times}\left[  0,1\right)  ^{m}\right)  }G_{l}\left(
\beta^{-1}\left(  s\right)  \right)  \exp\left(  2\pi i\left\langle
s,k\right\rangle \right)  ds\right\vert ^{2}\nonumber\\
&  =\sum_{l\in\mathbb{Z}^{m}}\int_{\mathbf{A}}\int_{\left[  0,1\right)  ^{m}%
}\left\vert G_{l}\left(  \lambda,t\right)  \right\vert ^{2}\left\vert \det
J_{\beta}\left(  \lambda,0\right)  \right\vert dt\text{ }d\lambda\\
&  =\sum_{l\in\mathbb{Z}^{m}}\int_{\mathbf{A}}\int_{\left[  0,1\right)  ^{m}%
}\left\vert F\left(  \lambda,t-l\right)  \right\vert ^{2}dt\text{ }d\mu\left(
\lambda\right)  . \label{topp}%
\end{align}
Equation (\ref{norm}) together with (\ref{topp}) gives
\[
\sum_{l\in\mathbb{Z}^{m}}\sum_{k\in\mathbb{Z}^{p}}\left\vert \int
_{\beta\left(  \mathbf{A\times}\left[  0,1\right)  ^{m}\right)  }G_{l}\left(
\beta^{-1}\left(  s\right)  \right)  \exp\left(  2\pi i\left\langle
s,k\right\rangle \right)  \text{ }ds\right\vert ^{2}=\left\Vert F\right\Vert
_{L^{2}\left(  \mathbf{A}\times\mathbb{R}^{m},d\mu\left(  \lambda\right)
\right)  }^{2}.
\]
Finally, we obtain
\[
\left\Vert F\right\Vert _{L^{2}\left(  \mathbf{A}\times\mathbb{R}^{m}%
,d\mu\left(  \lambda\right)  \right)  }=\left\Vert J_{\mathbf{A}}F\right\Vert
_{l^{2}\left(  \Gamma\right)  }.
\]
Now, since the set of continuous functions of compact support is dense in
$L^{2}\left(  \mathbf{A}\times\mathbb{R}^{m},d\mu\left(  \lambda\right)
\right)  $ and since $J$ defines an isometry on a dense set, then $J$ extends
uniquely to an isometry on $L^{2}\left(  \mathbf{A}\times\mathbb{R}^{m}%
,d\mu\left(  \lambda\right)  \right)  .$
\end{proof}

The proof given for Lemma \ref{isometry} can be easily modified to establish
the following result

\begin{lemma}
\label{unit}If $\beta\left(  \mathbf{A}\times\left[  0,1\right)  ^{m}\right)
$ has positive Lebesgue measure in $%
\mathbb{R}
^{p}$ and is equal to a fundamental domain of $\mathbb{Z}^{p}$ then
$J_{\mathbf{A}}$ defines an isometry on a dense subset of $L^{2}\left(
\mathbf{A}\times\mathbb{R}^{m},d\mu\left(  \lambda\right)  \right)  $ into
$l^{2}\left(  \Gamma\right)  $ which extends uniquely to a unitary map of
$L^{2}\left(  \mathbf{A}\times\mathbb{R}^{m},d\mu\left(  \lambda\right)
\right)  $ into $l^{2}\left(  \Gamma\right)  .$
\end{lemma}

\begin{remark}
\label{construction of admissible vectors}Suppose that $\beta\left(
\mathbf{A}\times\left[  0,1\right)  ^{m}\right)  $ has positive Lebesgue
measure in $%
\mathbb{R}
^{p}$ and is contained in a fundamental domain of $\mathbb{Z}^{p}.$ We have
shown that $J_{\mathbf{A}}\ $is an isometry. Now, let $\Phi$ be the orthogonal
projection of $l^{2}\left(  \Gamma\right)  $ onto the Hilbert space
$J_{\mathbf{A}}\left(  L^{2}\left(  \mathbf{A}\times\mathbb{R}^{m},d\mu\left(
\lambda\right)  \right)  \right)  $ and let $\kappa$ be the indicator sequence
of the singleton containing the identity element in $\Gamma.$ Identifying
$L^{2}\left(  \mathbf{A}\times\mathbb{R}^{m},d\mu\left(  \lambda\right)
\right)  $ with $\mathcal{P}\left(  \mathbf{H}_{\mathbf{A}}\right)  ,$ it is
clear that $\mathcal{P}^{-1}\left(  J_{\mathbf{A}}^{\ast}\left(  \Phi
\kappa\right)  \right)  \in\mathbf{H}_{\mathbf{A}}\subset L^{2}\left(
N\right)  $ and $L\left(  \Gamma\right)  \left(  \mathcal{P}^{-1}\left(
J_{\mathbf{A}}^{\ast}\left(  \Phi\kappa\right)  \right)  \right)  $ is a
Parseval frame for the band-limited Hilbert space $\mathbf{H}_{\mathbf{A}}$.
We remark that the vector $\kappa$ could be replaced by any other vector which
generates an orthonormal basis or a Parseval frame under the action of the
right regular representation of $\Gamma.$
\end{remark}

\section{Proof of Main Results}

\label{mainresults}

\subsection{Proof of Theorem \ref{Main 2}}

First, we observe that the right regular and left regular representations of
$\Gamma$ are unitarily equivalent (\cite{Folland}, Page $69$). To prove Part
$1,$ we appeal to Lemma \ref{isometry}, and Lemma \ref{intertwines}. Assuming
that $\beta\left(  \mathbf{A}\times\left[  0,1\right)  ^{m}\right)  $ has
positive Lebesgue measure in $\mathbb{R}^{p}$ and is contained in a
fundamental domain of $\mathbb{Z}^{p}$, the restriction of the representation
$\left(  L,\mathbf{H}_{\mathbf{A}}\right)  $ to the discrete group $\Gamma$ is
equivalent to a subrepresentation of the left regular representation of
$\Gamma.$ Appealing to Lemma \ref{sampling}, there exists a vector $\eta$ such
that $V_{\eta}\left(  \mathbf{H}_{\mathbf{A}}\right)  $ is a sampling space
with respect to $\Gamma.$ In fact, Remark
\ref{construction of admissible vectors} describes how to construct $\eta.$
For Part $2,$ Lemma \ref{isometry}, Lemma \ref{intertwines} together with the
assumption that $\beta\left(  \mathbf{A}\times\left[  0,1\right)  ^{m}\right)
$ is equal to a fundamental domain of $\mathbb{Z}^{p}$ imply that the
restriction of the representation $\left(  L,\mathbf{H}_{\mathbf{A}}\right)  $
to the discrete group $\Gamma$ is equivalent to the left regular
representation of $\Gamma.$ Finally, Remark
\ref{construction of admissible vectors} shows how to construct a vector
$\eta\in\mathbf{H}_{\mathbf{A}}$ such that $V_{\eta}^{L}\left(  \mathbf{H}%
_{\mathbf{A}}\right)  $ is a left-invariant subspace of $L^{2}\left(
N\right)  $ which is a sampling space with the interpolation property with
respect to $\Gamma.$ This completes the proof.

\subsection{Proof of Corollary \ref{Main1}}

For $s=\left(  s_{1},\cdots,s_{m}\right)  \in\mathbb{R}^{m},$ let $A\left(
s\right)  =s_{1}A_{1}+\cdots+s_{m}A_{m}\in\mathfrak{m}.$ Since the linear
operators $adA_{1},\cdots,adA_{m}$ are pairwise commutative and nilpotent,
since $e^{-adA\left(  s\right)  |\mathfrak{p}}$ is unipotent, there is a unit
vector which is an eigenvector for $e^{-adA\left(  s\right)  |\mathfrak{p}}$
with corresponding eigenvalue $1.$ So, it is clear that $\left\Vert \left[
e^{-adA\left(  s\right)  |\mathfrak{p}}\right]  ^{T}\right\Vert _{\infty}%
\geq1$ and
\[
\sup\left\{  \left\Vert \left[  e^{-adA\left(  s\right)  |\mathfrak{p}%
}\right]  ^{T}\right\Vert _{\infty}:s\in\mathbf{E}\right\}  \geq1
\]
for any nonempty $\mathbf{E}\subseteq\mathbb{R}^{m}.$ We recall again that%
\begin{equation}
\mathfrak{P}\left(  A(s)\right)  =\left[  e^{\left.  ad\left(  -\sum_{j=1}%
^{m}s_{j}A_{j}\right)  \right\vert \mathfrak{p}}\right]  ^{T}.\label{Mat}%
\end{equation}

\begin{lemma}
\label{betta 2}Let $\mathbf{E}$ be an open bounded subset of $\mathbb{R}^{m}.$
If $\varepsilon$ is a positive number satisfying
\[
\varepsilon\leq\delta=\left(  2\sup\left\{  \left\Vert \mathfrak{P}\left(
A(s)\right)  \right\Vert _{\infty}:s\in\mathbf{E}\right\}  \right)  ^{-1}%
\]
then $\beta\left(  \left(  \left(  -\varepsilon,\varepsilon\right)
^{\dim\Sigma}\cap\Sigma\right)  \times\mathbf{E}\right)  $ is open in
$\mathbb{R}^{p}$ and is contained in a fundamental domain of $\mathbb{Z}^{p}$.
\end{lemma}

\begin{proof}
Since the map $\beta$ is a diffeomorphism (see Lemma \ref{beta 1}) and since
the set
\[
\left(  \left(  -\varepsilon,\varepsilon\right)  ^{\dim\Sigma}\cap
\Sigma\right)  \times\mathbf{E}%
\]
is an open set in $\Sigma\times\mathbb{R}^{m}$, it is clear that its image
under the map $\beta$ is also open in $\mathbb{R}^{p}$. Next, it remains to
show that it is possible to find a positive real number $\delta$ such that if
$0<\varepsilon\leq\delta$ then $\beta\left(  \left(  \left(  -\varepsilon
,\varepsilon\right)  ^{\dim\Sigma}\cap\Sigma\right)  \times\mathbf{E}\right)
$ is an open set contained in a fundamental domain of $\mathbb{Z}^{p}.$ Let
$\lambda\in\Sigma.$ Then there exists a linear functional $f$ in the dual of
the ideal $\mathfrak{p}$ such that
\[
\iota\left(  \lambda\right)  =\left[
\begin{array}
[c]{c}%
f\\
0
\end{array}
\right]
\]
and
\begin{equation}
\iota\left(  \exp\left(  \sum_{j=1}^{m}s_{j}A_{j}\right)  \cdot\lambda\right)
=\left[
\begin{array}
[c]{c}%
\mathfrak{P}\left(  A(s)\right)  f\\
0
\end{array}
\right]  . \label{homog}%
\end{equation}
Moreover, it is worth noting that
\[
\left\Vert \exp\left(  \sum_{j=1}^{m}s_{j}A_{j}\right)  \cdot\lambda
\right\Vert _{\max}=\left\Vert \mathfrak{P}\left(  A(s)\right)  f\right\Vert
_{\max}.
\]
Let $\delta$ be a positive real number defined as follows:%
\begin{equation}
\delta=\left(  2\sup\left\{  \left\Vert \mathfrak{P}\left(  A(s)\right)
\right\Vert _{\infty}:s\in\mathbf{E}\right\}  \right)  ^{-1}. \label{delta}%
\end{equation}
If $f\in\left(  -\varepsilon,\varepsilon\right)  ^{\dim\Sigma}\subseteq\left(
-\delta,\delta\right)  ^{\dim\Sigma}$ and if $s\in\mathbf{E}$ then
\begin{align*}
\left\Vert \mathfrak{B}(A(s))f\right\Vert _{\infty}  &  \leq\left\Vert
f\right\Vert _{\max}\times\sup\left\{  \left\Vert \mathfrak{P}\left(
A(s)\right)  \right\Vert _{\infty}:s\in\mathbf{E}\right\} \\
&  =\frac{1}{2}\times\frac{\left\Vert f\right\Vert _{\max}}{\delta}.
\end{align*}
Now since $\left\Vert f\right\Vert _{\max}<\delta$, it follows that
\[
\left\Vert \mathfrak{P}\left(  A(s)\right)  f\right\Vert _{\max}<\frac{1}{2}.
\]
As a result,
\[
\beta\left(  \left(  \left(  -\varepsilon,\varepsilon\right)  ^{\dim\Sigma
}\cap\Sigma\right)  \times\mathbf{E}\right)  \subseteq\left(  -\frac{1}%
{2},\frac{1}{2}\right)  ^{p}%
\]
and clearly $\left(  -\frac{1}{2},\frac{1}{2}\right)  ^{p}$ is contained in a
fundamental domain of $%
\mathbb{Z}
^{p}.$
\end{proof}

Appealing to Lemma \ref{isometry}, and Lemma \ref{betta 2} the following is immediate

\begin{proposition}
\label{isometry delta}If
\[
0<\varepsilon\leq\delta=\frac{1}{2\sup\left\{  \left\Vert \mathfrak{P}\left(
A(s)\right)  \right\Vert _{\infty}:s\in\mathbf{E}\right\}  }%
\]
then $J_{\left(  -\varepsilon,\varepsilon\right)  ^{n-2m}\cap\Sigma}$ defines
an isometry between $L^{2}\left(  \left(  \left(  -\varepsilon,\varepsilon
\right)  ^{n-2m}\cap\Sigma\right)  \times\mathbb{R}^{m},d\mu\left(
\lambda\right)  \right)  $ and $l^{2}\left(  \Gamma\right)  $.
\end{proposition}

\subsubsection{Proof of Corollary \ref{Main1}}

Let $\delta$ be a positive number defined by
\begin{equation}
\delta=\frac{1}{2\sup\left\{  \left\Vert \mathfrak{P}\left(  A(s)\right)
\right\Vert _{\infty}:s\in\left[  0,1\right)  ^{m}\right\}  }.
\end{equation}
We want to show that for $\varepsilon\in\left(  0,\delta\right]  $ there
exists a band-limited vector
\[
\eta=\eta^{\varepsilon}\in\mathbf{H}_{\left(  -\varepsilon,\varepsilon\right)
^{n-2m}}%
\]
such that the Hilbert space $V_{\eta}^{L}\left(  \mathbf{H}_{\left(
-\varepsilon,\varepsilon\right)  ^{n-2m}}\right)  $ is a left-invariant
subspace of $L^{2}\left(  N\right)  $ which is a sampling space with respect
to $\Gamma.$ According to Lemma \ref{betta 2} the set
\[
\beta\left(  \left(  \left(  -\varepsilon,\varepsilon\right)  ^{\dim\Sigma
}\cap\Sigma\right)  \times\left[  0,1\right)  ^{m}\right)
\]
is open in $%
\mathbb{R}
^{p}$ and is contained in a fundamental domain of $\mathbb{Z}^{p}$. The
desired result follows immediately from Theorem \ref{Main 2}.

\subsection{Proof of Example \ref{dimfour} Part $1$}

The case of commutative simply connected, and connected nilpotent Lie group is
already known to be true. Thus, to prove this result, it remains to focus on
the non commutative algebras. According to the classification of
four-dimensional nilpotent Lie algebras \cite{Goze} there are three distinct
cases to consider. Indeed if $\mathfrak{n}$ is a non-commutative nilpotent Lie
algebra of dimension three, then $\mathfrak{n}$ must be isomorphic with the
three-dimensional Heisenberg Lie algebra. If $\mathfrak{n}$ is
four-dimensional then up to isomorphism either $\mathfrak{n}$ is the direct
sum of the Heisenberg Lie algebra with a one-dimensional algebra, or there is
a basis $Z_{1},Z_{2},Z_{3},A_{1}$ for $\mathfrak{n}$ with the following
non-trivial Lie brackets
\[
\left[  A_{1},Z_{2}\right]  =2Z_{1},\left[  A_{1},Z_{3}\right]  =2Z_{2}.
\]

\vskip0.2cm\noindent\textbf{Case} $1$ (The Heisenberg Lie algebra) Let $N$ be
the simply connected, connected Heisenberg Lie group with Lie algebra
$\mathfrak{n}$ which is spanned by $Z_{1},Z_{2},A_{1}$ with non-trivial Lie
brackets $\left[  A_{1},Z_{2}\right]  =Z_{1}.$We check that $N=PM$ where
$P=\exp\left(  \mathbb{R}Z_{1}+\mathbb{R}Z_{2}\right)  $ and $M=\exp\left(
\mathbb{R}A_{1}\right)  .$ Put
\[
\Gamma=\exp\left(  \mathbb{Z}Z_{1}\right)  \exp\left(  \mathbb{Z}Z_{2}\right)
\exp\left(  \mathbb{Z}A_{1}\right)  .
\]
It is easily checked that
\[
\mathbf{M}\left(  \lambda\right)  =\left[
\begin{array}
[c]{ccc}%
0 & 0 & 0\\
0 & 0 & -\lambda\left(  Z_{1}\right) \\
0 & \lambda\left(  Z_{1}\right)  & 0
\end{array}
\right]  .
\]
Next, since
\[
\mathbf{e}\left(  \lambda\right)  =\left\{
\begin{array}
[c]{c}%
\emptyset\text{ if }\lambda\left(  Z_{1}\right)  =0\\
\left\{  2,3\right\}  \text{ if }\lambda\left(  Z_{1}\right)  \neq0
\end{array}
\right.
\]
we obtain that $\mathbf{e=}\left\{  2,3\right\}  .$ It follows that
$\Omega_{\mathbf{e}}=\left\{  \lambda\in\mathfrak{n}^{\ast}:\lambda\left(
Z_{1}\right)  \neq0\right\}  .$ Next, the unitary dual of $N$ is parametrized
by $\Sigma=\left\{  \lambda\in\Omega_{\mathbf{e}}:\lambda\left(  Z_{2}\right)
=\lambda\left(  A_{1}\right)  =0\right\}  $ which we identify with the
punctured line: $\mathbb{R}^{\ast}.$ It is not hard to check that
\[
\delta^{-1}=2\sup\left\{  \left\Vert \left[
\begin{array}
[c]{cc}%
1 & 0\\
-s & 1
\end{array}
\right]  \right\Vert _{\infty}:s\in\left[  0,1\right)  \right\}  =4.
\]
So, there exists a band-limited vector $\eta\in\mathbf{H}_{\left(  -\frac
{1}{4},\frac{1}{4}\right)  }$ such that $V_{\eta}^{L}\left(  \mathbf{H}%
_{\left(  -\frac{1}{4},\frac{1}{4}\right)  }\right)  $ is a sampling space
with respect to $\Gamma.$

\vskip0.2cm \noindent\ To prove that the Heisenberg group admits sampling
spaces with the interpolation property with respect to $\Gamma,$ we claim that
the set
\[
B(1)=\beta\left(  \left(  -1,1\right)  \times\left[  0,1\right)  \right)
=\left\{  \left[
\begin{array}
[c]{c}%
f\\
-sf
\end{array}
\right]  :f\in\left(  -1,1\right)  ,s\in\left[  0,1\right)  \right\}
\]
is up to a null set equal to a fundamental domain of $\mathbb{Z}^{2}$ (see
illustration below)

\begin{center}

\includegraphics[scale=0.7]{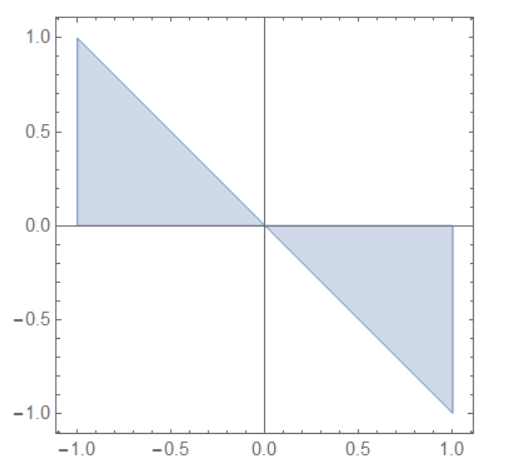}\newline The set $B(1)$
\end{center}

To prove this we write
\[
\beta\left(  \left(  -1,1\right)  \times\left[  0,1\right)  \right)
=\beta\left(  \left(  0,1\right)  \times\left[  0,1\right)  \right)  \cup
\beta\left(  \left(  -1,0\right)  \times\left[  0,1\right)  \right)  .
\]
Next, it is easy to check that
\[
\left(  \beta\left(  \left(  0,1\right)  \times\left[  0,1\right)  \right)
+\left[
\begin{array}
[c]{c}%
1\\
0
\end{array}
\right]  \right)  \cup\left(  \beta\left(  \left(  -1,0\right)  \times\left[
0,1\right)  \right)  +\left[
\begin{array}
[c]{c}%
0\\
1
\end{array}
\right]  \right)
\]
is up to a null set equal to the unit square $\left[  0,1\right)  ^{2}.$ Thus
the set $\beta\left(  \left(  -1,1\right)  \times\left[  0,1\right)  \right)
$ is up to a null set equal to a fundamental domain of $\mathbb{Z}^{2}.$
Appealing to Theorem \ref{Main 2}, the following result confirms the work
proved in \cite{Currey}\cite{oussa2}. There exists a band-limited vector
$\eta\in\mathbf{H}_{\left(  -1,1\right)  }$ such that $V_{\eta}^{L}\left(
\mathbf{H}_{\left(  -1,1\right)  }\right)  $ is a sampling space with respect
to $\Gamma$ which also enjoys the \textbf{interpolation property}.

\vskip0.2cm\noindent\textbf{Case} $2$ (Four-dimensional and step two) Assume
that $\mathfrak{n}$ is the direct sum of the Heisenberg Lie algebra with
$\mathbb{R}.$ That is $\mathfrak{n}$ which is spanned by $Z_{1},Z_{2}%
,Z_{3},A_{1}$ with non-trivial Lie brackets $\left[  A_{1},Z_{2}\right]
=Z_{1}.$ We check that
\[
\mathbf{M}\left(  \lambda\right)  =\left[
\begin{array}
[c]{cccc}%
0 & 0 & 0 & 0\\
0 & 0 & 0 & -\left(  Z_{1}\right) \\
0 & 0 & 0 & 0\\
0 & \lambda\left(  Z_{1}\right)  & 0 & 0
\end{array}
\right]
\]
and
\[
\mathbf{e}\left(  \lambda\right)  =\left\{
\begin{array}
[c]{c}%
\emptyset\text{ if }\lambda\left(  Z_{1}\right)  =0\\
\left\{  2,4\right\}  \text{ if }\lambda\left(  Z_{1}\right)  \neq0
\end{array}
\right.  .
\]
Fix $\mathbf{e=}\left\{  2,4\right\}  $ such that $\Omega_{\mathbf{e}%
}=\left\{  \lambda\in\mathfrak{n}^{\ast}:\lambda\left(  Z_{1}\right)
\neq0\right\}  $ and the unitary dual of $N$ is parametrized by
\[
\Sigma=\left\{  \lambda\in\Omega_{\mathbf{e}}:\lambda\left(  Z_{2}\right)
=\lambda\left(  A_{1}\right)  =0\right\}  .
\]
For any linear functional $\lambda\in\Sigma,$ the ideal spanned by
$Z_{1},Z_{2},Z_{3}$ is a polarization algebra subordinated to $\lambda$ and
\[
\delta=\left(  2\sup\left\{  \left\Vert \left[
\begin{array}
[c]{ccc}%
1 & 0 & 0\\
-s & 1 & 0\\
0 & 0 & 1
\end{array}
\right]  \right\Vert _{\infty}:s\in\left[  0,1\right)  \right\}  \right)
^{-1}=\frac{1}{4}.
\]

\vskip0.2cm\noindent\textbf{Case} $3$ (Four-dimensional and three step) Assume
that $\mathfrak{n}$ is a four-dimensional $Z_{1},Z_{2},Z_{3},A_{1}$ such that%
\[
\left[  A_{1},Z_{2}\right]  =2Z_{1},\left[  A_{1},Z_{3}\right]  =2Z_{2}.
\]
With respect to the ordered basis $Z_{1},Z_{2},Z_{3},$ we have
\[
\left.  \left[  adA_{1}\right]  \right\vert \mathfrak{p}=\left[
\begin{array}
[c]{ccc}%
0 & 2 & 0\\
0 & 0 & 2\\
0 & 0 & 0
\end{array}
\right]  \text{ and }\exp\left.  \left[  adA_{1}\right]  \right\vert
\mathfrak{p}=\left[
\begin{array}
[c]{ccc}%
1 & 0 & 2\\
0 & 1 & 0\\
0 & 0 & 1
\end{array}
\right]  .
\]
Next, we check that
\begin{align*}
\delta &  =\left(  2\sup\left\{  \left\Vert \left[
\begin{array}
[c]{ccc}%
1 & 0 & 0\\
-2s & 1 & 0\\
2s^{2} & -2s & 1
\end{array}
\right]  \right\Vert _{\infty}:s\in\left[  0,1\right)  \right\}  \right)
^{-1}\\
&  =\frac{1}{2}\left(  \max\left\{  1,1+2\left\vert s\right\vert
,1+2\left\vert s\right\vert +2\left\vert s\right\vert ^{2}:s\in\left[
0,1\right)  \right\}  \right)  ^{-1}=\frac{1}{10}.
\end{align*}
Indeed, the set%
\begin{align*}
\beta\left(  \left(  -\frac{1}{10},\frac{1}{10}\right)  ^{2}\times\left[
0,1\right)  \right)   &  =\left\{  \left[
\begin{array}
[c]{c}%
\lambda_{1}\\
-2s\lambda_{1}\\
2\lambda_{1}s^{2}+\lambda_{2}%
\end{array}
\right]  :\left(  \lambda_{1},\lambda_{1},s\right)  \in\left(  -\frac{1}%
{10},\frac{1}{10}\right)  ^{2}\times\left[  0,1\right)  \right\}  \\
&  \subset\left(  -\frac{1}{2},\frac{1}{2}\right)  ^{3}%
\end{align*}
is contained in a fundamental domain of $\mathbb{Z}^{3}.$ Thus, there exists a
band-limited vector $\eta\in\mathbf{H}_{\left(  -\frac{1}{10},\frac{1}%
{10}\right)  }$ such that $V_{\eta}^{L}\left(  \mathbf{H}_{\left(  -\frac
{1}{10},\frac{1}{10}\right)  }\right)  $ is a sampling space with respect to
\[
\Gamma=\exp\left(  \mathbb{Z}Z_{1}+\mathbb{Z}Z_{2}+\mathbb{Z}Z_{3}\right)
\exp\left(  \mathbb{Z}A_{1}\right)  .
\]

\subsection{Proof of Example \ref{dimfour} Part $2$ \label{OneParameter}}

Let $N$ be a simply connected, connected nilpotent Lie group with Lie algebra
spanned by $Z_{1},Z_{2},\cdots,Z_{p},A_{1}$ such that $\left.  \left[
adA_{1}\right]  \right\vert _{\mathfrak{p}}=A$ is a nonzero rational upper
triangular nilpotent matrix of order $p$ such that $e^{A}\mathbb{Z}%
^{p}\subseteq\mathbb{Z}^{p}$ and the algebra generated by $Z_{1},Z_{2}%
,\cdots,Z_{p}$ is commutative. Then $N$ is isomorphic to a semi-direct product
group $\mathbb{R}^{p}\rtimes\mathbb{R}$ with multiplication law given by
\[
\left(  x,t\right)  \left(  x^{\prime},t^{\prime}\right)  =\left(
x+e^{tA}x^{\prime},t+t^{\prime}\right)  .
\]
Clearly since $A$ is not the zero matrix then%
\[
\max\left\{  \mathrm{rank}\left(  \mathbf{M}\left(  \lambda\right)  \right)
:\lambda\in\mathfrak{n}^{\ast}\right\}  =2
\]
and the unitary dual of $N$ is parametrized by a Zariski open subset of
$\mathbb{R}^{p-1}.$ Finally, let
\[
\delta=\left(  2\times\sup\left\{  \left\Vert \sum_{k=0}^{m-1}\frac{\left(
-sA^{T}\right)  ^{k}}{k!}\right\Vert _{\infty}:s\in\left[  0,1\right)
\right\}  \right)  ^{-1}>0.
\]
For $\varepsilon\in\left(  0,\delta\right]  $ there exists a band-limited
vector $\eta=\eta^{\varepsilon}\in\mathbf{H}_{\left(  -\varepsilon
,\varepsilon\right)  ^{p-1}}$ such that the Hilbert space $V_{\eta}^{L}\left(
\mathbf{H}_{\left(  -\varepsilon,\varepsilon\right)  ^{p-1}}\right)  $ is a
left-invariant subspace of $L^{2}\left(  N\right)  $ which is a sampling space
with respect to $\Gamma.$

\subsection{Proof of Example \ref{dimfour} Part $3$ \label{steps}}

Let $N$ be a simply connected, connected nilpotent Lie group with Lie algebra
spanned by $Z_{1},Z_{2},\cdots,Z_{p},A_{1},\cdots,A_{m}$ where $p=m+1$ and the
matrix representation of $ad\left(  \sum_{k=1}^{m}t_{k}A_{k}\right)  $
restricted to $\mathfrak{p}$ is given by the following matrix of order $m+1$
\[
A\left(  t\right)  =\left.  \left[  ad\sum_{k=1}^{m}t_{k}A_{k}\right]
\right\vert \mathfrak{p}=m!\left[
\begin{array}
[c]{cccccc}%
0 & t_{1} & t_{2} & \cdots & t_{m-1} & t_{m}\\
& 0 & t_{1} & t_{2} & \ddots & t_{m-1}\\
&  & 0 & t_{1} & \ddots & \vdots\\
&  &  & 0 & \ddots & t_{2}\\
&  &  &  & \ddots & t_{1}\\
&  &  &  &  & 0
\end{array}
\right]  .
\]
We observe that
\[
\exp A\left(  t\right)  =\sum_{k=0}^{m+1}\frac{A\left(  t\right)  ^{k}}{k!}.
\]
Therefore, $N$ is a nilpotent Lie group of step $p=m+1$. Moreover, the unitary
dual of $N$ is parametrized by the manifold:
\[%
\begin{array}
[c]{c}%
\Sigma=\left\{  \lambda\in\mathfrak{n}^{\ast}:\lambda\left(  Z_{1}\right)
\neq0\right.  \\
\text{ }\left.  \text{and }\lambda\left(  Z_{k+1}\right)  =\lambda\left(
A_{k}\right)  =0\text{ for }1\leq k\leq m\right\}  \simeq\mathbb{R}^{\ast}%
\end{array}
\]
and the Plancherel measure is up to multiplication by a constant given by
$\left\vert \lambda\right\vert ^{m}d\lambda.$ Let
\[
r\left(  t\right)  =2\left\Vert \sum_{k=0}^{m+1}\frac{\left(  -A\left(
t\right)  ^{T}\right)  ^{k}}{k!}\right\Vert _{\infty}%
\]
be a function defined on $\mathbb{R}^{m}.$ The positive number $\delta$
described in Corollary \ref{Main1} is equal to
\[
\delta=\left(  \sup\left\{  r\left(  t\right)  :t\in\left[  0,1\right)
^{m}\right\}  \right)  ^{-1}.
\]
Thus, for $\varepsilon\in\left(  0,\delta\right]  $ there exists a
band-limited vector $\eta=\eta^{\varepsilon}\in\mathbf{H}_{\left(
-\varepsilon,\varepsilon\right)  }$ such that the Hilbert space $V_{\eta}%
^{L}\left(  \mathbf{H}_{\left(  -\varepsilon,\varepsilon\right)  }\right)  $
is a left-invariant subspace of $L^{2}\left(  N\right)  $ which is a sampling
space with respect to $\Gamma.$

\section{Construction of Other Sampling Sets}

\label{Othersamplingset}

In this section, we shall describe how to construct other sampling sets for
band-limited multiplicity-free spaces from a fixed given sampling set
$\Gamma.$ Assume that $\alpha$ is an automorphism of the Lie group $N$. Then
$\alpha$ induces the following unitary map $D:L^{2}\left(  N\right)
\rightarrow L^{2}\left(  N\right)  $ which is defined as follows: $\left[
Dh\right]  \left(  n\right)  =\Lambda\left(  \alpha\right)  ^{-1/2}h\left(
\alpha^{-1}\left(  n\right)  \right)  $ where
\[
\Lambda\left(  \alpha\right)  =\frac{d\left(  \alpha\left(  n\right)  \right)
}{dn}.
\]

\begin{lemma}
For any $x\in N,$ $DL\left(  x\right)  D^{-1}=L\left(  \alpha\left(  x\right)
\right)  .$
\end{lemma}

\begin{proof}
Let $h\in L^{2}\left(  N\right)  .$ Then $\left[  DL\left(  x\right)
D^{-1}h\right]  \left(  n\right)  =\Lambda\left(  \alpha\right)
^{-1/2}\left[  D^{-1}h\right]  \left(  x^{-1}\alpha^{-1}\left(  n\right)
\right)  $ and
\begin{align*}
\left[  DL\left(  x\right)  D^{-1}h\right]  \left(  n\right)   &  =h\left(
\alpha\left(  x^{-1}\right)  \alpha\left(  \alpha^{-1}\left(  n\right)
\right)  \right) \\
&  =h\left(  \alpha\left(  x^{-1}\right)  n\right)  =\left[  L\left(
\alpha\left(  x\right)  \right)  h\right]  \left(  n\right)  .
\end{align*}

\end{proof}

\begin{lemma}
Let $\mathbf{H}_{\mathbf{A}}$ be as defined in (\ref{HA}). The image of the
Hilbert space $\mathbf{H}_{\mathbf{A}}$ under the unitary map $D$ is
band-limited and is multiplicity-free.
\end{lemma}

\begin{proof}
It is well-known that (see Proposition $1.2,$ \cite{Grelaud}) for $n\in N$,
\[
\mathrm{Ind}_{P}^{N}\left(  \chi_{\lambda}\right)  \circ\alpha\left(
n\right)  =C\circ\left[  \mathrm{Ind}_{\alpha^{-1}\left(  P\right)  }%
^{N}\left(  \chi_{\lambda}\circ\alpha\right)  \left(  n\right)  \right]  \circ
C^{\ast}%
\]
for some unitary operator $C$ acting $L^{2}\left(  \mathbb{R}^{m}\right)  $
which is unique up to multiplication by a complex number of magnitude one
(according to Schur's lemma). Now, let $H$ be the group generated by the
automorphism $\alpha.$ Let $\left[  \mathrm{Ind}_{P}^{N}\left(  \chi_{\lambda
}\right)  \circ\alpha^{-1}\right]  $ be the class of irreducible
representations of $N$ which are equivalent to%
\[
\mathrm{Ind}_{P}^{N}\left(  \chi_{\lambda}\right)  \circ\alpha^{-1}.
\]
Then $H$ acts on the unitary dual of $N$ as follows
\[
\alpha\star\lambda=\left[  \mathrm{Ind}_{P}^{N}\left(  \chi_{\lambda}\right)
\circ\alpha^{-1}\right]  .
\]
Next, let $h\in\mathbf{H}_{\mathbf{A}}$ and let $\mathbf{w},\mathbf{v}\in
L^{2}\left(  \mathbb{R}^{m}\right)  .$ Then
\begin{align*}
\left\langle \left[  \mathcal{P}Dh\right]  \left(  \lambda\right)
\mathbf{w},\mathbf{v}\right\rangle  &  :=\left\langle \left[  \mathcal{P}%
Dh\right]  \left(  \sigma_{\lambda}\right)  \mathbf{w},\mathbf{v}\right\rangle
\\
&  =\Lambda\left(  \alpha\right)  ^{1/2}\int_{N}h\left(  n\right)
\left\langle \sigma_{\lambda}\left(  \alpha\left(  n\right)  \right)
\mathbf{w},\mathbf{v}\right\rangle dn\\
&  =\Lambda\left(  \alpha\right)  ^{1/2}\left\langle \left[  C\circ
\mathcal{P}h\left(  \alpha^{-1}\star\lambda\right)  \circ C^{\ast}\right]
\mathbf{w},\mathbf{v}\right\rangle .
\end{align*}
Since
\[
\left\langle \left[  \mathcal{P}Dh\right]  \left(  \lambda\right)
\mathbf{w},\mathbf{v}\right\rangle =\left\langle \Lambda\left(  \alpha\right)
^{1/2}\left[  C\circ\mathcal{P}h\left(  \alpha^{-1}\star\lambda\right)  \circ
C^{\ast}\right]  \mathbf{w},\mathbf{v}\right\rangle
\]
for arbitrary vectors $\mathbf{w},\mathbf{v}\in L^{2}\left(  \mathbb{R}%
^{m}\right)  ,$ it follows that
\[
\mathcal{P}\left(  Dh\right)  \left(  \lambda\right)  =\Lambda\left(
\alpha\right)  ^{1/2}C\circ\left(  \mathcal{P}h\right)  \left(  \alpha
^{-1}\star\lambda\right)  \circ C^{\ast}.
\]
Thus, the image of the Hilbert space $\mathbf{H}_{\mathbf{A}}$ under the
unitary map $D$ is band-limited and is multiplicity-free.
\end{proof}

\begin{proposition}
\label{Dilation}If $\beta\left(  \mathbf{A}\times\left[  0,1\right)
^{m}\right)  $ is contained in a fundamental domain of $\mathbb{Z}^{p}$ then
there exists a Parseval frame of the type $\left\{  L\left(  \gamma\right)
z:\gamma\in\alpha\left(  \Gamma\right)  \right\}  $ for $D\mathbf{H}%
_{\mathbf{A}}$ and there exists a vector $\eta\in D\mathbf{H}_{\mathbf{A}}$
such that $V_{\eta}\left(  D\mathbf{H}_{\mathbf{A}}\right)  $ is a sampling
space with respect to $\alpha\left(  \Gamma\right)  .$
\end{proposition}

\begin{proof}
Let us suppose that $\beta\left(  \mathbf{A}\times\left[  0,1\right)
^{m}\right)  $ is contained in a fundamental domain of $\mathbb{Z}^{p}.$ Let
$\left\{  L\left(  \gamma\right)  h:\gamma\in\Gamma\right\}  $ be a Parseval
frame for $\mathbf{H}_{\mathbf{A}}.$ Then $\left\{  D\left(  L\left(
\gamma\right)  h\right)  :\gamma\in\Gamma\right\}  $ is a Parseval frame for
the Hilbert space $D\mathbf{H}_{\mathbf{A}}$ and since
\[
DL\left(  \gamma\right)  D^{-1}=L\left(  \alpha\left(  \gamma\right)  \right)
\]
it follows that
\[
\left\{  L\left(  \alpha\left(  \gamma\right)  \right)  Dh:\gamma\in
\Gamma\right\}  =\left\{  L\left(  \gamma\right)  Dh:\gamma\in\alpha\left(
\Gamma\right)  \right\}
\]
is a Parseval frame for $D\mathbf{H}_{\mathbf{A}}.$ The fact that there exists
a vector $\eta\in D\mathbf{H}_{\mathbf{A}}$ such that $V_{\eta}\left(
D\mathbf{H}_{\mathbf{A}}\right)  $ is a sampling space with respect to
$\alpha\left(  \Gamma\right)  $ is due to Proposition $2.60$ and Proposition
$2.54$ \cite{Fuhr cont}.
\end{proof}

\section{Concluding Observations}

Let us conclude this work by exhibiting an example which does not belong to
the class of groups presented described in Condition \ref{cond}. Let
$\mathfrak{n}$ be a five-dimensional nilpotent Lie algebra with basis spanned
by $Z_{1},Z_{2},Z_{3},A_{1},A_{2}$ such that $\left[  A_{1},Z_{3}\right]
=Z_{2},\left[  A_{2},Z_{3}\right]  =Z_{1}.$ Next, let $\mathfrak{p}$ be the
ideal spanned by $Z_{1},Z_{2},Z_{3}$ and let $\mathfrak{m}$ be the ideal
spanned by $A_{1},A_{2}.$ Let $N$ be a simply connected, connected nilpotent
Lie group with Lie algebra $\mathfrak{n.}$ Then $N=P\rtimes M$ is a metabelian
nilpotent Lie group, and its dual is parametrized by the set
\[
\Sigma=\left\{  \zeta_{1}Z_{1}^{\ast}+\zeta_{2}Z_{2}^{\ast}+\alpha_{1}%
A_{1}^{\ast}:\zeta_{1}\neq0\right\}
\]
which is a cross-section for all coadjoint orbits in the Zariski open set
\[
\Omega=\left\{  \zeta_{1}Z_{1}^{\ast}+\zeta_{2}Z_{2}^{\ast}+\zeta_{3}%
Z_{3}^{\ast}+\alpha_{1}A_{1}^{\ast}+\alpha_{2}A_{2}^{\ast}:\zeta_{1}%
\neq0\right\}  .
\]
The coadjoint orbits in $\Omega$ are two-dimensional manifolds, the ideal
$\mathfrak{p}$ is not a polarization for any linear functional $\lambda$ in
$\Omega.$ In fact it is properly contained in one. Indeed for any linear
functional in the cross-section $\Sigma,$ a polarization algebra subordinated
to $\lambda$ must be a four-dimensional algebra. For example the set
\[
\left\{  Z_{1},Z_{2},Z_{3},A_{1}-\frac{\zeta_{2}}{\zeta_{1}}A_{2}\right\}
\]
spans a polarization subordinated to $\zeta\in\Omega.$ If there exists a
subalgebra of $\mathfrak{n}$ which is a constant polarization, then such an
algebra must be four-dimensional. However, there is no four-dimensional
subalgebra of $\mathfrak{n}$ which is a maximal commutative ideal. Thus, the
results proved in this work do not apply to this group. To the best of our
knowledge, it is an open question if the results of this paper extend to
nilpotent Lie groups which do not belong to the class of groups considered
here. This problem will be the focus of a future investigation.

\end{document}